\documentclass[amssymb,amsfonts,refcheck,12pt,verbatim,righttag]{amsart}
\usepackage{color}
\usepackage{graphicx}
\usepackage{caption}
\usepackage{tikz}

\def\colon{{:}\;}

\def\R {\Bbb R}
 \numberwithin{equation}{section} 
\newcommand{\beq}{\begin{equation}}
\newcommand{\eeq}{\end{equation}}

\newcommand{\ben}{\begin{eqnarray}}
\newcommand{\een}{\end{eqnarray}}

\newcommand{\bet}{\begin{eqnarray*}}
\newcommand{\eet}{\end{eqnarray*}}

\newtheorem{thm}{Theorem}[section]
\newtheorem{lem}[thm]{Lemma}

\newtheorem{prop}[thm]{Proposition}
\newtheorem{cor}[thm]{Corollary}
\newtheorem{de}[thm]{Definition}

\theoremstyle{definition}
\newtheorem{rem}[thm]{Remark}
\newtheorem{ex}[thm]{Example}

\def\R {\Bbb R}
\def\N {\Bbb N}

\usepackage[colorlinks,citecolor=red,pagebackref,hypertexnames=false]{hyperref}

\theoremstyle{plain}

\begin{document}
\baselineskip 15.2pt

\title
{On arithmetic sums of fractal sets  in ${\Bbb R}^d$}

\author{De-Jun FENG}
\address{
Department of Mathematics\\
The Chinese University of Hong Kong\\
Shatin,  Hong Kong\\
}
\email{djfeng@math.cuhk.edu.hk}

\author{Yu-Feng Wu}
\address{
Department of Mathematics\\
The Chinese University of Hong Kong\\
Shatin,  Hong Kong\\
}
\email{yfwu@math.cuhk.edu.hk}

\thanks{}
\date{}

\begin{abstract}
A compact set $E\subset \R^d$ is said to be arithmetically thick if there exists a positive integer $n$ so that the $n$-fold arithmetic sum of $E$ has non-empty interior.
We prove the arithmetic thickness of $E$, if $E$ is uniformly non-flat, in the sense that there exists $\epsilon_0>0$ such that  for $x\in E$ and $0<r\leq {\rm diam}(E)$, $E\cap B(x,r)$ never stays $\epsilon_0r$-close to a hyperplane in $\R^d$.   Moreover,  we prove the arithmetic thickness for several classes of fractal sets, including  self-similar sets,  self-conformal sets  in $\R^d$ (with $d\geq 2$) and self-affine sets in $\R^2$ that do not lie in a hyperplane, and certain self-affine sets in $\R^d$ (with $d\geq 3$) under specific assumptions.
\end{abstract}

\maketitle
\setcounter{section}{0}
\section{Introduction}
\label{S1}
\setcounter{equation}{0}

For   $E_1,\ldots, E_n\subset \R^d$, the arithmetic sum of $E_i$'s is defined as
$$
E_1+\cdots+E_n=\{x_1+\cdots+x_n:\; x_i\in E_i \mbox{ for }1\leq i\leq n\}.
$$
For convenience, we also write  $\bigoplus_{i=1}^nE_i=E_1+\cdots+E_n$. A compact set $E\subset \R^d$ is said to be {\it arithmetically thick} if there exists a positive integer $n$ so that the $n$-fold arithmetic sum $\oplus_nE$ of $E$ has non-empty interior, where
$$
\oplus_nE:=\{x_1+\cdots+x_n:\; x_i\in E \mbox{ for }1\leq i\leq n\}.
$$
 As a generalized version of the Steinhaus theorem, the arithmetic sum of any two measurable subsets of $\R^d$ with positive Lebesgue measure always contains non-empty interior (see e.g. \cite{Kestelman1947}). As a direct consequence, each compact subset of $\R^d$ with positive Lebesgue measure is arithmetically thick.   A natural question arises  how to check the arithmetic thickness for a given compact set with zero Lebesgue measure.  It looks quite unlikely that there exists a simple checkable criterion which works for all compact sets in this question. In this paper, we aim to prove the arithmetic thickness for  some concrete sets that appear in geometric measure theory and fractal geometry.

In the literature there have been many works on or related to the above question in the  case  $d=1$ (see e.g. \cite{Astels2000, CHM1997, CHM2002, FraserHowroydYu2019, Hall1947, MoreiraYoccoz2001, newhouse, PalisTakens1993, RossiShmerkin2018, Solomyak1997, Takahashi2019a, Takahashi2019b}). One of the main concerns is whether the arithmetic sum of two or more Cantor sets contains an interval or has large fractal dimensions. Here by a Cantor set we mean a compact  subset of $\R$ that is perfect and nowhere dense.  In \cite{newhouse} Newhouse introduced a notion of thickness for  Cantor sets (which nowadays is called {\it Newhouse thickness}) and proved that for any two Cantor sets $A$ and $B$, the sum $A+B$ has non-empty interior if $\tau_N(A)\tau_N(B)\geq 1$, where $\tau_N(\cdot)$ denotes the Newhouse thickness (see also \cite[p.~61]{PalisTakens1993} for the definition). In  \cite{Astels2000, CHM2002}, it was proved that, among other things,   a Cantor set in $\R$ is arithmetically thick if  it has ratios of dissection bounded away from zero. As a direct consequence, every non-singleton self-similar set (and more generally, every non-singleton self-conformal set satisfying  the bounded distortion property)  in $\R$ is arithmetically thick, since it contains a Cantor subset which has ratios of dissection bounded away from zero. The reader is referred to Section \ref{S-2} for the relevant definitions of self-similar and self-conformal sets.

So far as we know, there have been only a few results for the case $d\geq 2$. In \cite{NikodemPales2010} Nikodem and P\'{a}les proved a  result on the arithmetic sums of homogeneous fractal sets in Banach spaces which, applied to Euclidean spaces, yields that if $E$ is the self-similar set generated by a homogeneous iterated function system $\{\rho x+a_i\}_{i=1}^\ell$ in $\R^d$,    then there exists $n$ so that $\oplus_nE=n\;{\rm conv}(F)$, where $F:=\{a_i/(1-\rho):\; i=1,\ldots, \ell\}$ and ${\rm conv}(F)$ stands for the convex hull of $F$. In particular, it implies the fact that $E$ is arithmetically thick provided that $E$ is not contained in a hyperplane.  Later this fact was  independently proved by Oberlin and Oberlin in \cite{Oberlin2017}. Recently,   Banakh, Jab\l{}o\'{n}ska and Jab\l{}o\'{n}ski \cite{BJJ2019} proved that, under mild assumptions,  the arithmetic sum of $d$ many compact connected sets in $\R^d$  has non-empty interior. As a consequence, every compact connected set in $\R^d$ not lying in a hyperplane  is arithmetically thick. As related works,
in \cite{SimonTaylor2017, SimonTaylor2020} Simon and Taylor gave some sufficient conditions so that the arithmetic sums of planar sets and curves have positive Lebesgue measure or non-empty interior.

Before stating our main results,  we first introduce the concept of  thickness for compact subsets of $\R^d$.  For $x\in \R^d$ and $r>0$, let $B(x,r)$ denote the closed ball centred at $x$ of radius $r$.  For $F\subset \R^d$, let ${\rm diam}(F)$ and ${\rm conv}(F)$ denote the diameter and the convex hull of $F$, respectively.

\begin{de}
\label{de-1.1} Let $E$ be a compact set in $\R^d$. The  thickness of $E$, denoted by $\tau(E)$, is the largest number $c\in[0,1]$ such that for each $x\in E$ and  $0<r\leq {\rm diam}(E)$, there exists $y=y(x,r)\in \R^d$ satisfying
${\rm conv}(B(x, r)\cap E)\supset B(y, cr).$
\end{de}

Returning back to the case when $d=1$, our definition of thickness is different from  that of Newhouse thickness. Nevertheless, it is easily  checked that for a Cantor set
$A$ in $\R$, $\tau(A)>0$ if and only if $\tau_N(A)>0$.

It is worth pointing out that our definition of thickness is closely related to the notion of uniform non-flatness introduced by David in \cite{David2004} (see also \cite{BishopJones1997, Jones1990}) and the notion of hyperplane diffuseness introduced by Broderick {\it et al.}~in \cite{BFKRW2012}. Recall that a set $E\subset \R^d$ is said to be {\it uniformly non-flat} if there exists $\epsilon_0>0$ such that  for $x\in E$ and $0<r\leq {\rm diam}(E)$, $E\cap B(x,r)$ never stays $\epsilon_0r$-close to a hyperplane in $\R^d$. Meanwhile, a set $E\subset \R^d$ is said to be {\it hyperplane diffuse} if there exist $\rho=\rho_E>0$ and $c>0$ such that for any $x\in E$ and $0<r<\rho$, $E\cap B(x,r)$ is not contained in the $cr$-neighborhood of any hyperplane in $\R^d$.
  It is easy to check that a compact set $E\subset \R^d$ is uniformly non-flat (resp.~hyperplane diffuse)  if and only if it has positive thickness.

Our first main result of this paper is the following.

\begin{thm}\label{thm: sum of thick sets} Let $E_1, \ldots, E_n$ be compact sets in $\R^d$ such that $\tau(E_i)\geq c>0$ for $1\leq i\leq n$. Then $\bigoplus_{i=1}^nE_i$ has non-empty interior provided that $n> 2^{11}c^{-3}+1$.
\end{thm}

As a corollary, each compact subset of $\R^d$ with positive thickness is arithmetically thick.  Since a self-similar set $E\subset \R^d$ has positive thickness if and only if $E$ is not contained in a hyperplane in $\R^d$ (see Lemma \ref{lem-self-similar}),   we obtain the following.

\begin{cor}
\label{cor-similar}
Every self-similar set in $\R^d$ not lying in a hyperplane is arithmetically thick.
\end{cor}

 Our next  result extends the above result  to all self-conformal sets in $\R^d$ with $d\geq 2$.

\begin{thm}\label{main thm: self-conformal} Let $d\geq 2$.  Suppose that  $E$ is  a self-conformal set generated by a conformal iterated function system on  $\R^d$. Then
  $E$ is arithmetically thick if and only if $E$ is not contained in a hyperplane in $\R^d$.
\end{thm}

Finally we  investigate the sums of self-affine sets (see Section \ref{S-2} for the definition).   First we introduce some definitions.
\begin{de}
\begin{itemize}
\item[(i)]
A finite tuple $(M_1,\ldots, M_k)$ of $d\times d$ real matrices is said to be {\it irreducible}
if there is no non-zero proper linear subspace V of $\R^d$ such that $M_iV\subset V$ for all $1\leq i\leq k$.
\item[(ii)] A $d\times d$ real matrix $M$ is said to
have a {\it simple dominant eigenvalue} if  $M$ has a simple eigenvalue $\lambda$ (i.e.~an eigenvalue with algebraic multiplicity $1$) so that $|\lambda|$  is greater than the magnitude of any other eigenvalue of $M$.
\end{itemize}
\end{de}

Now we are ready to state our result on self-affine sets.

\begin{thm}\label{thm: self-affine,commulative Ai} Let $E$ be the attractor of an affine iterated function system  $\Phi=\{\phi_i(x)=T_ix+a_i\}_{i=1}^{\ell}$
on $\R^d$ with $d\geq 2$. Suppose that $E$ is not contained in a hyperplane in $\R^d$. Then $E$ is arithmetically thick if either one of the following conditions is fulfilled:
\begin{itemize}
\item[(i)] $T_iT_j=T_jT_i$ for all $1\leq i, j\leq \ell$;
\item[(ii)] $(T_1,\ldots, T_\ell)$ is irreducible, and the multiplicative semigroup generated by $T_1$,$\ldots$, $T_\ell$ contains an element which has a simple dominant eigenvalue;
\item[(iii)] $d=2$.
 \end{itemize}
\end{thm}

    We emphasize that under the settings of Theorems \ref{main thm: self-conformal}-\ref{thm: self-affine,commulative Ai}, a self-conformal set (resp, self-affine set) in $\R^d$ not lying in a hyperplane may have zero thickness.  So we can not directly apply Theorem \ref{thm: sum of thick sets} to prove Theorems \ref{main thm: self-conformal}-\ref{thm: self-affine,commulative Ai}.

     It is worth pointing out that if a closed set $E\subset \R^d$ supports a Borel probability measure $\mu$ whose Fourier transform has a power decay at infinity (i.e. $|\widehat{\mu}(\xi)|\leq C|\xi|^{-\alpha}$ for some constants $C, \alpha>0$), then $E$ is arithmetically thick. This follows from the well-known fact that when $n\alpha>d/2$, the $n$-fold convolution $\mu^{*n}$ of $\mu$ (which is supported on $\oplus_nE$) is absolutely continuous (with $L^2$ density),  so $\oplus_nE$ has positive Lebesgue measure and 
     $\oplus_{2n}E$ has non-empty interior. Nevertheless, it is a difficult question to determine whether a given fractal set can support a Borel probability measure whose Fourier transform has power decay at infinity. Recently, Li and Sahlsten (\cite[Theorem 2]{LiSahlsten2019}) proved that for an affine iterated function system  $\{T_ix+a_i\}_{i=1}^{\ell}$ on $\R^d$, if its attractor is not a singleton, then under the irreducibility and certain additional algebraic assumptions on the semigroup generated by $T_1, \ldots, T_\ell$,  every fully supported self-affine measure associated with the IFS has power delay in its Fourier transform. We remark that these assumptions are stronger than that in part (ii) of Theorem \ref{thm: self-affine,commulative Ai}.   Under some weaker assumptions (which are similar to that in part (ii) of Theorem \ref{thm: self-affine,commulative Ai}, but the irreducibility is replaced by the strong irreducibility), Li and Sahlsten showed that the Fouirer transform of every fully supported self-affine measure tends to $0$  at infinity; see \cite[Theorem 1]{LiSahlsten2019}.

The organization of the paper is as follows: In Section \ref{S-2}, we give the definitions of iterated function systems and self-similar (resp. self-affine, self-conformal) sets.   In Section \ref{S2}, we give some elementary lemmas which play key roles in our proofs of the main results. In Section \ref{S3},  we prove Theorem \ref{thm: sum of thick sets}.   In Section \ref{S4}, we prove Theorem \ref{main thm: self-conformal}. Theorem \ref{thm: self-affine,commulative Ai} is proved in Section \ref{S5}.  In Section \ref{S8}, we prove a special result on the arithmetic sum of rotation-free self-similar sets, which partially generalises  the aforementioned result of   Nikodem and P\'{a}les \cite{NikodemPales2010}. In Section \ref{S-8}, we give some final remarks and questions.

\section{Preliminaries on iterated function systems}
\label{S-2}
In mathematics,  iterated function system (IFS) is a basic scheme to generate fractal sets. By definition, an {\it IFS} on a closed subset $X$ of $\R^d$ is a finite family $\Phi=\{\phi_i: X\to X\}_{i=1}^\ell$ of uniformly contracting mappings on $X$, in the sense that there exists $0<c<1$ such that $|\phi_i(x)-\phi_i(y)|\leq c|x-y|$ for all $x,y\in X$ and $1\leq i \leq \ell$.   The {\it attractor} of $\Phi$  is the unique non-empty compact set $K\subset X$ so that
$$
K=\bigcup_{i=1}^\ell \phi_i(K).
$$
 The IFS $\Phi$ induces a coding map $\pi: \{1,\ldots, \ell\}^\N\to K$, which is given by
 \begin{equation}
 \label{e-coding}
 \pi(x)=\lim_{n\to \infty} \phi_{x_1}\circ \cdots\circ \phi_{x_n}(z_0)
 \end{equation}
 where $z_0$ is any fixed point in $X$.  The map $\pi$ is surjective and it is  independent of the choice of $z_0$.
 The reader is referred to \cite{Hutchinson1981, Falconer2003} for  more information about IFS.

A mapping $f \colon \R^d \to \R^d$ is said to be \emph{affine} if $f(x) = Tx + a$ for all $x \in \R^d$, where $T$ is a $d \times d$ matrix and $a \in \R^d$.  It is easy to see that an affine map $f$ is invertible if and only if its linear part $T$  is non-singular, moreover $f$ is strictly contracting if and only if its linear part has operator norm
$\| T \|$ strictly less than $1$.  A non-empty compact set $E \subset \R^d$ is called \emph{self-affine} if $E = \bigcup_{i=1}^\ell f_i(E)$, where $\{ f_i \}_{i=1}^\ell$ is an \emph{affine IFS}, i.e.~a finite collection of uniformly contracting invertible affine mappings on $\R^d$. Moreover, $E$ is called \emph{self-similar} if all the $f_i$'s are similitudes.

  Let $U\subset \R^d$ be a  connected open set. A $C^1$ map $\phi: U\to\R^d$ is said to be {\em conformal} if $\|\phi'(x)y\|=\|\phi'(x)\|\cdot\|y\|\neq 0$ for all  $x\in U$ and $y\in \R^d, y\neq 0$. The well-known theorem of Liouville \cite{Li} states that when $d\geq 3$, every $C^1$ conformal map $\phi:\; U\to \R^d$ is the restriction to $U$ of a M{\"o}bius transformation in $\R^d$. Recall that a M{\"o}bius transformation $\psi$ in $\R^d$, $d\geq 3$, is of the form
\begin{equation}
\label{e-form}
\psi(x)=b+\frac{\alpha A(x-a)}{\|x-a\|^{\epsilon}},
\end{equation}
where $a, b\in\R^d$,  $\alpha\in\R$, $\epsilon\in\{0,2\}$ and $A$ is a $d\times d$ orthogonal matrix.

  We say that $\Phi=\{\phi_i: X\to X\}_{i=1}^{\ell}$ is a {\em conformal IFS} on a compact set $X\subset\R^d$ if each $\phi_i$ extends to an injective contracting conformal map $\phi_i: U\to \phi_i(U)\subset U$ on a bounded connected open set $U\supset X$.   The attractor $E$ of  $\Phi$ is called  the {\em self-conformal set} generated by $\Phi$.  Let $U_1$ be a connected open set such that $X\subset U_1\subset \overline{U_1}\subset U$. It is well-known that when $d\geq 2$, $\Phi$ satisfies the {\em bounded distortion property} (BDP) on $U_1$: there exists $L\geq 1$ such that for every $n$  and every word $I=i_1\ldots i_n$ over the alphabet $\{1,\ldots, \ell\}$,
\begin{equation}\label{eq:bdp0}L^{-1}\leq \frac{\|\phi_I'(x)\|}{\|\phi_I'(y)\|}\leq L, \ \  \forall x, y\in U_1,
\end{equation}
where $\phi_I:=\phi_{i_1}\circ\cdots\circ\phi_{i_n}$. This follows from the (generalized) Koebe distortion theorem (see e.g. \cite[Theorem 7.16]{conway}) when $d=2$, and from the form of M{\"o}bius transformations when $d\geq 3$;  see \eqref{e-form}.


\section{Some elementary lemmas}
\label{S2}
In this section, we prove some elementary lemmas which will be used in the proofs of the main results.
For $A\subset \R^d$, let ${\rm conv}(A)$ denote the convex hull of $A$.

\begin{lem}
\label{lem-3.1}
Let $A=\{a_1,\ldots, a_n\}\subset \R^d$ and $\epsilon\in [0, 1/n]$. Then $${\rm conv}(A)+{\rm conv}(\epsilon A)= {\rm conv}(A)+\epsilon A.$$
\end{lem}
\begin{proof}
It suffices to show that ${\rm conv}(A)+{\rm conv}(\epsilon A)\subset{\rm conv}(A)+\epsilon A$. To see this, let $x\in {\rm conv}(A)$ and $y\in {\rm conv}(\epsilon A)$. Then there exist probability vectors $(p_1,\ldots, p_n)$ and { $(q_1,\ldots, q_n)$} such that $x=\sum_{i=1}^n p_ia_i$ and $y=\sum_{i=1}^n \epsilon q_i a_i$. Choose $j\in \{1,\ldots, n\}$ such that
$p_j=\max\{p_i:\; i=1,\ldots, n\}$. Clearly $p_j\geq 1/n\geq \epsilon$. Define a vector $(\tilde{p}_1,\ldots, \tilde{p}_n)$ by
$$
\tilde{p}_i=\left\{
\begin{array}{cc}
p_i+\epsilon q_i &\mbox{ if }i\neq j\\
p_j+\epsilon q_j -\epsilon & \mbox{ if }i=j
\end{array}
\right..
$$
It is direct to check that $(\tilde{p}_1,\ldots, \tilde{p}_n)$ is a probability vector, hence
\begin{eqnarray*}
x+y-\epsilon a_j&=&(p_j+\epsilon q_j -\epsilon) a_j +\sum_{1\leq i\leq n,\; i\neq j} (p_i+\epsilon q_i)a_i \\
&=& \sum_{i=1}^n \tilde{p}_i a_i \in {\rm conv}(A).
\end{eqnarray*}
That is, $x+y\in {\rm conv}(A)+\epsilon a_j\subset {\rm conv}(A)+\epsilon A$. Since $x,y$ are arbitrarily taken from ${\rm conv}(A)$ and ${\rm conv}(\epsilon A)$ respectively, it follows that ${\rm conv}(A)+{\rm conv}(\epsilon A)\subset {\rm conv}(A)+\epsilon A$ and we are done.
\end{proof}

Let $\|\cdot\|$ denote the standard Euclidean norm in $\R^d$. For $A\subset \R^d$, let ${\rm diam}(A)$ be the diameter of $A$.
\begin{lem}
\label{lem-dist}
Let $A\subset \R^d$ be bounded. Suppose  $\|a\|\geq R$ for every $a\in A$. Then for each $z\in {\rm conv}(A)$,
$$
\|z\|\geq  R-{\rm diam}(A)^2/(2 R).
$$ \end{lem}
\begin{proof}
{ We may assume that $R>{\rm diam}(A)/\sqrt{2}$, otherwise we have nothing to prove. }

Let $z\in {\rm conv}(A)$. Then $z=\sum_{i=1}^n p_ia_i$ for some $a_1,\ldots, a_n\in A$ and $p_1,\ldots, p_n\geq 0$ with $p_1+\cdots +p_n=1$. Let $\langle\cdot,\cdot\rangle$ denote the standard inner product in $\R^d$. Then
\begin{eqnarray*}
\|z\|^2&=&\left\langle \sum_{i=1}^n p_ia_i, \sum_{j=1}^n p_ja_j\right\rangle\\
&=&\left(\sum_{i=1}^n p_i^2 \|a_i\|^2\right)+ \left(\sum_{i\neq j} p_ip_j \langle a_i, a_j\rangle\right)\\
&=&\left(\sum_{i=1}^n p_i^2 \|a_i\|^2\right)+\left(\sum_{i\neq j} p_ip_j \left(\|a_i\|^2+\|a_j\|^2-\|a_i-a_j\|^2\right)/2\right)\\
&\geq& \left(\sum_{i=1}^n p_i^2 R^2\right)+\left(\sum_{i\neq j} p_ip_j \left(2 R^2-{\rm diam}(A)^2\right)/2\right)\\
&=& R^2-\left(\sum_{i\neq j} p_ip_j\right) {\rm diam}(A)^2/2\\
&\geq & R^2-{\rm diam}(A)^2/2.
\end{eqnarray*}
{ Hence
$
\|z\|\geq R\sqrt{1-{\rm diam}(A)^2/(2R^2)}\geq R (1-{\rm diam}(A)^2/(2R^2)).
$
}
\end{proof}

{For $x\in \R^d$, let $B(x,r)$ be the closed ball of radius $r$ centred at $x$. For $x\in\R^d$ and $F\subset \R^d$, let $d(x, F)$ be the distance from $x$ to $F$.}
\begin{cor}
\label{cor-4.3} Let $A\subset \R^d$ be bounded. Suppose  ${\rm conv}(A)\supset B(y,r)$ for some $y\in \R^d$ and $r>0$. Then for all $R> {\rm diam}(A)^2/r$ and  $z\in \R^d$,
\begin{equation}
\label{e-4}
B(z, R)+A\supset B(z, R)+ B(y,r/2).
\end{equation}

\end{cor}
\begin{proof}
Let $R>{\rm diam}(A)^2/r$.  Since ${\rm conv}(A)\supset B(y,r)$,  we have ${\rm diam}(A)\geq 2r$. It follows that $R> {\rm diam}(A)^2/r\geq 2{\rm diam}(A)\geq 4r$.

To prove that \eqref{e-4} holds for all $z\in \R^d$, it suffices to show that \eqref{e-4}  holds for $z=0$. Write $X=B(0, R)+A$. Since $R>2{\rm diam}(A)$ and $X\supset B(a, R)$ for each $a\in A$, we see that ${\rm interior}(X)\supset {\rm conv}(A)$. In particular, $y\in {\rm interior}(X)$. Hence to prove that
$X\supset B(0, R)+B(y,r/2)=B(y, R+r/2)$, it is enough to show that $d(y,\partial X)>R+r/2$, where $\partial X$ stands for the boundary of $X$.

Fix $x\in \partial X$. In what follows we show that $d(x,y)>R+r/2$. Recall that $X\supset U(0, R)+A$, which is the open $R$-neighborhood of $A$.  It follows that $d( x, A)\geq R$ (otherwise $x\in {\rm interior}(X)$). Applying Lemma \ref{lem-dist} to the set $(A-x)$ yields that for every $z\in {\rm conv}(A-x)$,
$$
\|z\|\geq R-{\rm diam}(A)^2/(2R)> R-r/2,
$$
where the second inequality follows from the assumption that $R>{\rm diam}(A)^2/r$.
 Equivalently,
\begin{equation}
\label{e-5}
d(x, {\rm conv}(A))> R-r/2.
\end{equation}
Let $L=L_{xy}$ be the line segment connecting the points $x$ and $y$. Since $y\in {\rm conv}(A)$, by \eqref{e-5} $L$ has length $>R-r/2$.  Since $R-r/2>R/2>{\rm diam}(A)={\rm diam} ({\rm conv}(A))$,  the length of $L$ is larger than ${\rm diam}({\rm conv}(A))$.  It follows that $L$ is not contained in the interior of ${\rm conv}(A)$. In particular, this implies that $L\cap \partial ({\rm conv}(A)) \neq \emptyset$. Take $z\in L\cap \partial ({\rm conv}(A))$.  Now $L=L_{xz}\cup L_{zy}$. Notice that  $d(x,z)> R-r/2$ by  \eqref{e-5}, and $d(z,y)\geq r$ since $B(y,r)\subset {\rm conv}(A)$. Hence $L$ has length $> R-r/2 +r=R+r/2$. That is, $d(x,y)> R+r/2$.  This completes the proof.
\end{proof}
\begin{lem}
\label{lem-3.4}
Let $A\subset \R^d$. Suppose that $B(z,r)\subset {\rm conv}(A)$ for some $z\in \R^d$ and $r>0$. Then for any $0<\delta<r$ and $F\subset \R^d$ with $V_\delta(F)\supset A$, we have
$$
U(z,r-\delta):=\{y\in \R^d:\; \|y-z\|<r-\delta\}\subset {\rm conv}(F).
$$
Here  $V_\delta(F)=\{y\in \R^d:\; d(y, F)<\delta\}$.
\end{lem}
\begin{proof}
First observe that  $U(0, \delta)+{\rm conv}(F)={\rm conv}(V_\delta(F))$, which  can be verified directly.
Since $V_\delta(F)\supset A$, it follows that
\begin{equation}
\label{e-10}
U(0,\delta)+{\rm conv}(F)\supset {\rm conv}(A)\supset B(z,r).
\end{equation}

In what follows we prove $U(z,r-\delta)\subset {\rm conv}(F)$ by using contradiction.  Suppose this is not true. Then there exists $x\in U(z,r-\delta)$ so that $x\not\in {\rm conv}(F)$. By the hyperplane separation theorem (see e.g. \cite[Theorem 11.3]{rochkafellar}), there is a hyperplane passing through $x$ so that  ${\rm conv}(F)$ entirely lies on the one side of the hyperplane. Equivalently, there exists a unit vector $v\in \R^d$ and $c\in \R$ so that $\langle x, v\rangle =c$ and
$\langle u, v\rangle\leq c$ for all $u\in {\rm conv}(F)$.

Set $y=x+\delta v$. Then $\|y-z\|\leq \|y-x\|+\|z-x\|<\delta+(r-\delta)=r$. Hence $y\in U(z,r)$.  Now notice that $\langle y, v\rangle=\langle x, v\rangle+\langle \delta v, v\rangle=c+\delta$, and for any $w\in U(0,\delta)$ and $u\in {\rm conv}(F)$,
$$
\langle u+w, v\rangle=\langle u, v\rangle+\langle w, v\rangle\leq c+\delta.
$$
 This means that  there is a hyperplane separating the point $y$ and the set  $U(0,\delta)+{\rm conv}(F)$. By \eqref{e-10}, this hyperplane also separates the point $y$ and the ball $B(z,r)$.  It leads to a contradiction, since $y$ is an interior point of $B(z,r)$.
\end{proof}

\begin{lem}
\label{lem-self-similar}
Let $E$ be a self-similar set in $\R^d$. Then $E$  has positive thickness if and only if $E$ is not contained in a hyperplane in $\R^d$.
\end{lem}
\begin{proof}
The result was pointed out in \cite[p.~330]{BFKRW2012} without a proof.
For  the reader's convenience, we provide a proof.

 The `only if' part is trivial so we only need to prove the `if' part.   To this end, assume that $E$ is not contained in a hyperplane.  Then ${\rm conv}(E)$ contains a ball, say $B(x_0, r_0)$.  Let $\{\phi_i\}_{i=1}^\ell$ be a generating IFS of $E$ and let $\rho_i$ denote the contraction ratio of $\phi_i$, $i=1,\ldots, \ell$.  Set $\rho_{\rm min}=\min_{1\leq i\leq
 \ell}\rho_i$.

 Let $x\in E$ and $0<r\leq {\rm diam}(E)$.   Then there exists $(\omega_n)_{n=1}^\infty\in \{1,\ldots, \ell\}^\N$ such that $$\{x\}=\bigcap_{n=1}^\infty \phi_{\omega_1}\circ\cdots\circ \phi_{\omega_n}(E).$$
 Moreover, there exists $n\in \N$ such that
 \begin{equation}
 \label{e-gt1}
 \rho_{\omega_1}\cdots\rho_{\omega_n}{\rm diam}(E)<r\leq \rho_{\omega_1}\cdots\rho_{\omega_{n-1}}{\rm diam}(E).
 \end{equation}
 It follows that $ \rho_{\omega_1}\cdots\rho_{\omega_n}{\rm diam}(E)\geq \rho_{\rm min}r$ and so
 \begin{equation*}
 \label{e-gt2}
 \rho_{\omega_1}\cdots\rho_{\omega_n}\geq \rho_{\rm min}({\rm diam}(E))^{-1}r.
 \end{equation*}
 By \eqref{e-gt1}, $B(x,r)\supset \phi_{\omega_1}\circ \cdots \circ \phi_{\omega_n}(E)$. Hence
 \begin{align*}
 {\rm conv}(E\cap B(x,r))&\supset {\rm conv}(\phi_{\omega_1}\circ \cdots \circ \phi_{\omega_n}(E))\\
 &= \phi_{\omega_1}\circ \cdots \circ \phi_{\omega_n} ({\rm conv}(E))\\
 &\supset \phi_{\omega_1}\circ \cdots \circ \phi_{\omega_n} (B(x_0, r_0))\\
 &= B(y, \rho_{\omega_1}\cdots\rho_{\omega_n}r_0)\\
 &\supset B(y, \rho_{\rm min}r_0({\rm diam}(E))^{-1}r),
 \end{align*}
 where $y=\phi_{\omega_1}\circ \cdots \circ \phi_{\omega_n}(x)$.  Hence by definition, $\tau(E)\geq  \rho_{\rm min}r_0({\rm diam}(E))^{-1}>0$.
\end{proof}

In the rest of this section, following \cite{BFKRW2012} we give an equivalent condition for a compact set in $\R^d$ to have positive thickness. We first introduce the following.

\begin{de}\label{de: centred microsets}Let $E$ be a non-empty compact set in $\R^d$. A compact set $F$ is said to be a centred microset of $E$ if $F$ is a limit point of a sequence of compact sets
\[\frac{1}{r_n}((B(x_n, r_n)\cap E)-x_n)\]
in the Hausdorff metric, where $x_n\in E$, $r_n>0$ and $\lim_{n\to \infty}  r_n=0$.
\end{de}

The above definition is a slight modification of the notion of microset introduced by Furstenberg in \cite{Furstenberg2008}.
Now we  state the following equivalent condition for positive thickness,  which will be used in the proofs of Theorems \ref{main thm: self-conformal}-\ref{thm: self-affine,commulative Ai}.

\begin{lem}\label{prop: no mircoset in hyperlane implies thickness}{\cite[Lemma 4.4]{BFKRW2012}}
Let $E$ be a non-empty compact set in $\R^d$. Then $\tau(E)>0$ if and only if no centred microset of $E$ is contained in a proper linear subspace of $\R^d$.
\end{lem}

\section{Proof of Theorem \ref{thm: sum of thick sets}}
\label{S3}
In this section, we prove  Theorem \ref{thm: sum of thick sets}.  As the proof is rather long and a bit technical,  before giving the detailed  arguments we would like to illustrate briefly the rough strategy of our proof.
Basically we will construct, for each pair $(i, k)$ with $1\leq i\leq n$ and $k\in {\Bbb N}$, a finite family  ${\mathcal F}_{i, k}$ of closed balls of radius $t_k$ with   $t_k\searrow  0$ so that there exists $s_k\searrow 0$ such that
$
\bigcup_{B\in {\mathcal F}_{i,k}}B\subset V_{s_k}(E_i)
$
and  $H_k:=\bigoplus_{i=1}^n \left(\bigcup_{B\in {\mathcal F}_{i,k}}B\right)$ is monotone increasing in $k$, where $V_\epsilon(E)$ stands for the $\epsilon$-neighborhood of $E$. Then we have
$$
\bigoplus_{i=1}^nV_{s_k}(E_i)\supset H_k\supset H_{k-1}\supset\cdots\supset H_1.
$$
Taking $k\to \infty$ yields that $\bigoplus_{i=1}^nE_i\supset H_1$, which concludes the theorem since $H_1$ has non-empty interior.

 Although the above strategy is very simple, the involved constructions are relatively delicate.  Below  we first give a geometric property of  compact sets with positive thickness.

\begin{lem}\label{lem: a property of thick set} Let $E$ be a compact set in $\R^d$ with $\tau(E)\geq c>0$.  Let  $N$ be the integral part of $\left(\frac{4+c}{c}\right)^d$. Then  for every $x\in E$ and  $0<r\leq {\rm diam}(E)$, there exist  $z\in\R^d$ and $y_1, \ldots, y_N\in E\cap B(x, r)$ such that
\begin{equation*}
\label{e-4.0}
{\rm conv}(\{y_1, \ldots, y_N\})\supset B(z, cr/2).
\end{equation*}
\end{lem}

\begin{proof}Fix $x\in E$ and $0<r\leq {\rm diam}(E)$. By the definition of $\tau(E)$, there exists $z\in\R^d$ such that
\begin{equation}
\label{e-4.1}
{\rm conv}(E\cap B(x, r))\supset B(z, cr).
\end{equation}

Let $N_0$ be the largest integer such that there exist disjoint open balls $U(y_1, cr/4)$, $\ldots$ , $U(y_{N_0}, cr/4)$ in $\R^d$ with centers $y_i\in E\cap B(x, r)$.
Since the balls $U(y_i, cr/4)$ are disjoint and contained in $U(x, r+cr/4)$, a standard volume argument yields that
\[N_0\leq \left(\frac{4+c}{c}\right)^d\]
 and so $N_0\leq N$.
 Meanwhile the maximality of $N_0$ implies that
\begin{equation}
\label{e-4.2}
E\cap B(x, r)\subset\bigcup_{i=1}^{N_0}U(y_i, cr/2).
\end{equation}
 To see this, suppose on the contrary that $y\not\in \bigcup_{i=1}^{N_0}U(y_i, cr/2)$ for some $y\in E\cap B(x, r)$.
Then $|y-y_i|\geq cr/2$ and so $U(y, cr/4)\cap U(y_i, cr/4)=\emptyset$ for each $1\leq i\leq N_0$, which contradicts the maximality of $N_0$. Hence \eqref{e-4.2} holds.

 Next we apply Lemma \ref{lem-3.4} to show that
${\rm conv}(\{y_1,\ldots, y_{N_0}\})\supset B(z, cr/2)$. For this purpose, set $A=E \cap B(x,r)$ and
$F=\{y_1,\ldots, y_{N_0}\}$. Then by \eqref{e-4.1}-\eqref{e-4.2},  ${\rm conv}(A)\supset B(z, cr)$ and $V_{cr/2}(F)\supset A$.  Applying  Lemma \ref{lem-3.4} to $A$ and $F$ (in which we  replace $r$ by $cr$ and take  $\delta=cr/2$)  yields  ${\rm conv}(F)\supset U(z,cr/2)$. Since ${\rm conv}(F)={\rm conv}(\{y_1,\ldots, y_{N_0}\})$ is compact, it follows that ${\rm conv}(\{y_1,\ldots, y_{N_0}\})\supset B(z, cr/2)$, as desired.

Finally taking $y_j=y_{N_0}$ for $N_0<j\leq N$, we obtain  that ${\rm conv}(\{y_1,\ldots, y_{N}\})\supset B(z, cr/2)$. This completes the proof of the lemma.
\end{proof}
Now we are ready to prove Theorem \ref{thm: sum of thick sets}.

\begin{proof}[Proof of Theorem \ref{thm: sum of thick sets}]
Set $r_0=\min\{{\rm diam}(E_i): 1\leq i\leq n\}$  and  $\rho=c/4$. Let $N$ be the integral part of $\left(\frac{4+c}{c}\right)^d$, as given  in Lemma \ref{lem: a property of thick set}. For convenience, write $\Sigma_*:=\bigcup_{k=1}^{\infty}\Sigma_k$, where $\Sigma_k:=\{1,\ldots, N\}^k$. Set $|J|=k$ for $J\in \Sigma_k$. Below for each $1\leq i\leq n$, we construct inductively a family of balls $\{B_I^i\}_{I\in\Sigma_*}$.

To illustrate our construction,  fix $i\in \{1,\ldots,  n\}$. Choose any point from $E_i$ and write it as $x_{\emptyset}^i$.  Set $B_{\emptyset}^i=B(x_{\emptyset}^i, r_0)$. Since $\tau(E_i)\geq c$, according to Lemma \ref{lem: a property of thick set}, we can pick points $x_1^i,\ldots, x_N^i\in E_i\cap B(x_{\emptyset}^i, r_0/2)$
and $z_{\emptyset}^i\in\R^d$ so that
\[{\rm conv}(\{x_j^i: 1\leq j\leq N\})\supset B(z_{\emptyset}^i, cr_0/4).\]
Set
\[B_j^i=B(x_j^i, \rho r_0), \ \  j=1,\ldots, N.\]
Then we have defined well the balls $\{B_I^i\}_{I\in\Sigma_1}$.

 Next we continue the construction  process by induction.  Suppose we have constructed well the family of  balls $\{B_J^i: J\in\Sigma_k\}$ with centers $\{x_J^i\}_{J\in\Sigma_k}$ for some integer $k\geq 1$. Then  by Lemma \ref{lem: a property of thick set},  for each $J\in\Sigma_k$ we can pick points $x_{J1}^i, \ldots, x_{JN}^i$ in $E_i\cap B(x_J^i, \rho^kr_0/2)$ such that
\begin{equation}
\label{e-4.1e}
{\rm conv}\{x_{J1}^i, \ldots, x_{JN}^i\}\supset B(z_J^i, \rho^kcr_0/4)
\end{equation}
for some $z_J^i\in\R^d$.  Clearly $z_J^i\in {\rm conv}\{x_{J1}^i, \ldots, x_{JN}^i\}\subset  B(x_J^i, \rho^k r_0/2)$ and so
\begin{equation}
\label{e-4.2e}
|x_J^i-z_J^i|\leq \rho^k r_0/2.
\end{equation}
 Defining  $B_{Jj}^i=B(x_{Jj}^i, \rho^{k+1}r_0)$ for $1\leq j\leq N$, we complete the construction of the balls $\{B_J^i: J\in \Sigma_{k+1}\}$.
According to  the above construction,
 $$
 \bigcup_{j=1}^NB_{Jj}^i=\bigcup_{j=1}^N B(x_{Jj}^i, \rho^{|J|+1}r_0)\subset B(x_J^i, \rho^{|J|} r_0)=B_J^i ,
$$
since $x_{J1}^i, \ldots, x_{JN}^i\in  B(x_J^i, \rho^{|J|}r_0/2)$ and $\rho<1/2$.  By induction, we can construct well the whole family of balls $\{B_I^i\}_{I\in\Sigma_*}$, together with the family $\{z_J^i\}_{J\in \Sigma_*}$ of points in $\R^d$.  { See Figure \ref{fig1} for a rough illustration of  the above construction.}
\begin{figure}
\centering
\begin{tikzpicture}[scale=0.8]
\draw (0,0) circle[radius=6];
\draw (0,0) circle[radius=3];
\draw (-2.5,0.5)--(-1, -2)--(1.8,-0.8)--(-2.5,0.5);
\draw [->] (0,0)--(1,2.828);
\draw [->] (0,0)--(3,5.1962);
\draw (-0.6,-1) circle[radius=0.685];
\draw [->](-0.6,-1)--(-1.26,-1.1);
\foreach \Point in {(0,0), (-2.5,0.5), (-1,-2), (1.8,-0.8), (-0.6,-1)}{
    \node at \Point [circle,fill,inner sep=1pt]{};}
\node[above] at (-0.7, -1.06) {$z_J^i$};
 \node at (3,-4) {$B_J^i$};
\node [left, above] at (-0.2,0) {$x_J^i$};
\node [above] at (-2.45,0.4) {$x_{J1}^i$};
\node [below] at (-1,-1.9) {$x_{J2}^i$};
\node [right] at (-0.5, -2.2) {$B_{J2}^i$};
\node [above] at (1.8,-0.85) {$x_{J3}^i$};
\node [above] at (1.95,-0.2) {$B_{J3}^i$};

\node[left] at (0.5,1.414) {$\frac{\rho^kr_0}{2}$};
\node[right] at (2,3.5) {$\rho^kr_0$};
\draw[->](-1,-1) to [out=90,in=90] (-3.2,-1.1);
\node[left] at (-3.1,-1.1) {$\frac{c\rho^kr_0}{4}$};
\draw (-2.5,0.5) circle[radius=0.685];

\node [above] at (-1.5, 0.6) {$B_{J1}^i$};

\draw (-1,-2) circle[radius=0.685];
\draw(1.8,-0.8) circle[radius=0.685];
\draw[->] (1.8,-0.8)--(2.5,-0.8);
\draw[->] (2.2,-0.8) to [out=90,in=90](3.5, -0.8);
\node [right] at (3.4,-0.8) {$\rho^{k+1}r_0$};
\end{tikzpicture}
\caption{An illustration of the balls $B_{Jj}^{i}$.}
\label{fig1}
\end{figure}
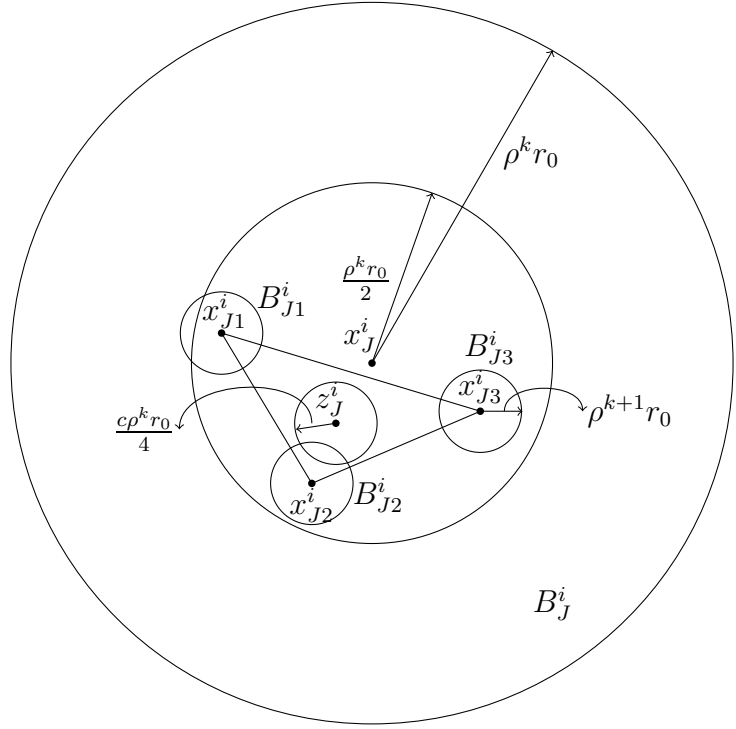

 Now we present some properties of the constructed $\{B_I^i\}_{I\in\Sigma_*}$ and  $\{z_J^i\}_{J\in \Sigma_*}$. Let $1\leq i\leq n$ and  $J\in\Sigma_*$.   Let $\epsilon>0$. By \eqref{e-4.2e}, for $1\leq j\leq N$,    $$|x^i_{Jj}-z^i_{Jj}|\leq \rho^{|J|+1}r_0/2<\rho^{|J|+1}r_0/2+\epsilon.$$  Due to the above inequality and \eqref{e-4.1e},  we apply Lemma \ref{lem-3.4}
 (in which taking $A=\{x_{J1}^i,\ldots, x_{JN}^i\}$, $z=z_J^i$, $r=\rho^{|J|}cr_0/4$, $F=\{ z_{J1}^i, \ldots, z_{JN}^i\}$ and $\delta=\rho^{|J|+1}r_0/2+\epsilon$)
  to obtain that
\[{\rm conv}\{z_{J1}^i, \ldots, z_{JN}^i\}\supset U(z_J^i, \rho^{|J|}cr_0/4-\rho^{|J|+1}r_0/2-\epsilon)=U(z_J^i, \rho^{|J|}cr_0/8-\epsilon).\]
As $\epsilon>0$ is arbitrarily taken, we have
\[{\rm conv}\{z_{J1}^i, \ldots, z_{JN}^i\}\supset U(z_J^i, \rho^{|J|}cr_0/8).\]
Since ${\rm conv}\{z_{J1}^i, \ldots, z_{JN}^i\}$ is compact, it follows that
\begin{equation}
\label{e-4.3e}
{\rm conv}\left\{z_{J1}^i, \ldots, z_{JN}^i\right\}\supset B\left(z_J^i, \rho^{|J|}cr_0/8\right).
\end{equation}
Meanwhile, since $|z_{Jj}^i-x_{Jj}^i|\leq \rho^{|J|+1}r_0/2$ and $|x_{Jj}^i-x_J^i|\leq \rho^{|J|}r_0/2$ for $j=1,\ldots, N$, it follows that  $|z_{Jj}^i-x_J^i|\leq  \rho^{|J|}r_0$ and thus
\begin{equation}
\label{e-4.4e}
{\rm diam}\left({\rm conv}\left(\left\{z_{J1}^i,\ldots, z_{JN}^i\right\}\right)\right)\leq 2\rho^{|J|}r_0.
 \end{equation}

Next assume that $n> 2^{11}c^{-3}+1$. We claim that for every $k\in \N$ and any $J_1,\ldots, J_n\in\Sigma_k$,
\begin{equation}
\label{e-4.5e}
 \bigoplus_{i=1}^n\left(\bigcup_{j=1}^NB\left(z_{J_ij}^i, \rho^{k+1}cr_0/16\right)\right)\supset \bigoplus_{i=1}^nB\left(z_{J_i}^i, \rho^{k}cr_0/16\right).
\end{equation}
 To prove the claim, we first introduce some notation. Write for brevity that $D_0=\emptyset$,
$F_n=\emptyset$,
\begin{align*}
D_\ell&:=\bigoplus_{i=1}^\ell B\left(z_{J_i}^i, \rho^{k}cr_0/16\right) \mbox{  for $\ell=1,\ldots, n$ and}\\
F_\ell&:= \bigoplus_{i=\ell+1}^n\left(\bigcup_{j=1}^NB\left(z_{J_ij}^i, \rho^{k+1}cr_0/16\right)\right) \mbox{  for $\ell=0, 1,\ldots, n-1$}.
\end{align*}
Then \eqref{e-4.5e} is simply the statement that  $F_0\supset D_n$.
In what follows we shall prove that for $\ell=0,1,\ldots, n-1$,
\begin{equation}
\label{e-4.6e}
D_\ell+F_\ell\supset D_{\ell+1}+F_{\ell+1},
\end{equation}
which implies that $F_0=D_0+F_0\supset D_n+F_n=D_n$ and so \eqref{e-4.5e} holds.

To prove \eqref{e-4.6e}, fix $\ell\in \{0,1,\ldots, n-1\}$. Notice that
\begin{equation}
\label{e-2020-1.1}
F_\ell=\left(\bigcup_{j=1}^NB\left(z_{J_{\ell+1}j}^{\ell+1}, \rho^{k+1}cr_0/16\right)\right)+F_{\ell+1}.
\end{equation}
Write $A=\left\{z_{J_{\ell+1}j}^{\ell+1}:\; j=1,\ldots, N\right\}$. By \eqref{e-4.3e}-\eqref{e-4.4e},
$$
{\rm conv}(A)\supset B\left(z_{J_{\ell+1}}^{\ell+1}, \rho^k cr_0/8\right)\quad \mbox{ and }\quad \mbox{diam} (A) \leq 2\rho^kr_0.
$$
Applying Corollary \ref{cor-4.3} (in which we take $y=z_{J_{\ell+1}}^{\ell+1}$ and $r=\rho^k cr_0/8$) yields
\begin{equation}
\label{e-4.7e}
B(z, R)+A\supset B(z, R)+B\left(z_{J_{\ell+1}}^{\ell+1}, \rho^k cr_0/16\right)
\end{equation}
for any $z\in\R^d$, provided that $R>8 \mbox{diam} (A)^2/(\rho^k cr_0)$.  Notice that $D_\ell+F_{\ell+1}$ is the union of finitely many balls, say $B_1,\ldots, B_m$, and each of them is of radius
$$R_\ell:=\ell \rho^kcr_0/16+(n-1-\ell)\rho^{k+1}cr_0/16\geq (n-1)\rho^{k+1}cr_0/16.$$
Since $n> 2^{11}c^{-3}+1$ and $\mbox{diam} (A) \leq 2\rho^kr_0$, a direct check shows that $$R_\ell>(n-1)\rho^{k+1}cr_0/16\geq  8 \mbox{diam} (A)^2/(\rho^k cr_0),$$ and hence by \eqref{e-4.7e},
 $B_i+A\supset B_i+B\left(z_{J_{\ell+1}}^{\ell+1}, \rho^k cr_0/16\right)$ for $i=1,\ldots, m$.   Taking union over $i$ yields that
 $$
 D_\ell+F_{\ell+1}+A\supset D_\ell+F_{\ell+1}+B\left(z_{J_{\ell+1}}^{\ell+1}, \rho^k cr_0/16\right)= D_{\ell+1}+F_{\ell+1},
 $$
 from which we see that
 $$D_\ell+F_\ell\supset D_\ell+F_{\ell+1}+A\supset D_{\ell+1}+F_{\ell+1}
 $$
 (where the first inclusion is due to \eqref{e-2020-1.1}) and so
 \eqref{e-4.6e} follows. This completes the proof of \eqref{e-4.5e}.

 Taking union over   $(J_{1},\ldots, J_n)\in (\Sigma_k)^n$ in \eqref{e-4.5e} yields that
 $$
H_{k+1}\supset H_k,
 $$
 where  $H_k:=\bigoplus_{i=1}^n\left(\bigcup_{J\in\Sigma_k}B\left(z_J^i, \rho^kcr_0/16\right)\right)$.   Since $|z_J^i-x_J^i|\leq \rho^kr_0/2$ for any $J\in \Sigma_k$, and $\{x_{J}^i:\; J\in \Sigma_k\}\subset E_i$,  it follows that
 $$\overline{V}_{\rho^kr_0}(E_i):=\{y\in \R^d:\; d(y, E_i)\leq \rho^kr_0\}\supset \bigcup_{J\in\Sigma_k}B\left(z_J^i, \rho^kcr_0/16\right)$$
 and thus
 \[\overline{V}_{\rho^kr_0}(E_1)+\cdots+\overline{V}_{\rho^kr_0}(E_n)\supset H_k\supset\cdots\supset H_1.\]
Since the sets $E_i$ are compact, letting $k\to\infty$ yields
\[E_1+\cdots+E_n\supset H_1.\]
This completes the proof of the theorem, for $H_1$ has non-empty interior.
\end{proof}

\section{Proof of Theorem \ref{main thm: self-conformal}}
\label{S4}

Throughout this section, let $\Phi=\{\phi_i: X\to X\}_{i=1}^{\ell}$ be  an IFS on a compact set $X\subset\R^d$ with $d\geq 2$ so that
 each $\phi_i$ extends to an injective contracting conformal map $\phi_i: U\to \phi_i(U)\subset U$ on
a bounded connected open set $U\supset X$.
Furthermore we assume that
the attractor of $\Phi$, written as $E$, is not a singleton.  Let $\Sigma_*$ denote the collection of all finite words (including the empty word) over the alphabet $\{1,\ldots, \ell\}$, that is,
 $\Sigma_*=\bigcup_{n=0}^\infty\{1,\ldots,\ell\}^n$.

The following lemma characterizes when $E$ has positive thickness.
\begin{lem}
\label{lem-4.1}
Under the above setting, we have $\tau(E)>0$ unless one of the following cases occurs:
\begin{itemize}
\item[(i)] $d=2$ and $E$ is contained in a simple analytic curve in $\R^2$.
\item[(ii)] $d\geq 3$,   $E$ is contained in a hyperplane in $\R^d$ or  a $(d-1)$-dimensional sphere in $\R^d$.
\end{itemize}
\end{lem}
\begin{proof}
The result was pointed out in \cite[p.~330]{BFKRW2012} without a proof. It was also implicitly proved in \cite[Theorem 2.3]{Kaenmaki2003} and
\cite[Theorem 1.2]{MayerUrbanski2003}  in slightly different contexts.
For  the reader's convenience, we provide a detailed proof.

Since $\Phi$ satisfies the bounded distortion property on $U$ (cf.~\eqref{eq:bdp0}), it is known (see, e.g. \cite[Lemma 2.2, Corollary 2.3]{patzschke97}) that there exists an open connected set $V$  such that $E\subset V\subset U$,
$\bigcup_{i=1}^\ell\phi_i(V) \subset V$, and there is a constant $C>0$ so that for any $x, y\in V$ and $I\in \Sigma_*$,
\begin{equation}\label{eq: bilipschitz0}
C^{-1}\alpha_I\|x-y\|\leq \|\phi_I(x)-\phi_I(y)\|\leq C\alpha_I\|x-y\|,
\end{equation}
where $\alpha_I:=\sup_{x\in V}\|\phi'_I(x)\|$.   As a consequence,
\begin{equation}\label{eq: bound diam}
C^{-1}\alpha_I{\rm diam}(E)\leq {\rm diam}(\phi_I(E))\leq C\alpha_I{\rm diam}(E),  \ \ \forall I\in\Sigma_*.
\end{equation}
We may assume that $C$ is large enough so that
\begin{equation}\label{eq-1}
\alpha_I\leq \alpha_{\widehat{I}}\leq C \alpha_I \mbox{ for all } I\in\Sigma_*,
\end{equation}
where $\widehat{I}$ stands for the word obtained from $I$ by  dropping the last letter of $I$.

To prove the lemma, we need to show that if $\tau(E)=0$, then either (i) or (ii) occurs.  To this end, assume that $\tau(E)=0$.
By Proposition \ref{prop: no mircoset in hyperlane implies thickness}, $E$ has a centred microset lying in a proper linear subspace of $\R^d$.
That is, there exist  $x_n\in E$,  $r_n>0$, $n=1,2,\ldots,$  with $\lim_{n\to \infty} r_n=0$ such that
\begin{equation}\label{eq: conformal1}
\frac{1}{r_n}((B(x_n,r_n)\cap E)-x_n)\to F \  \text{as } n\to \infty
\end{equation}
in the Hausdorff metric, where $F$ is a compact set contained in a $(d-1)$-dimensional linear subspace $W$ of $\R^d$.

For each $n\in \N$, take $I_n\in \Sigma_*$ such that
$$
x_n\in \phi_{I_n}(E),\quad \phi_{I_n}(E)\subset B(x_n, r_n) \; \mbox{ and } \;\phi_{\widehat{I_n}}(E)\not\subset B(x_n, r_n).
$$
Clearly $\mbox{diam}(\phi_{I_n}(E))\leq 2r_n$ and $\mbox{diam}(\phi_{\widehat{I_n}}(E))>r_n$.  Combining these two inequalities with \eqref{eq: bound diam}-\eqref{eq-1} yields $(2C)^{-1}\alpha_{I_n}\mbox{diam}(E)\leq r_n\leq C^2\alpha_{I_n}\mbox{diam}(E)$, and so
\begin{equation}
\label{eq-2}
C^{-2}(\mbox{diam}(E))^{-1}\leq \alpha_{I_n}/r_n\leq 2C(\mbox{diam}(E))^{-1}.
\end{equation}

Define $\psi_n:\R^d\to\R^d$ by $\psi_n(x)=(x-x_n)/r_n$ for $n\geq 1$. Write $f_n=\psi_n\circ \phi_{I_n}$. Clearly $f_n$ is conformal and injective for each $n$.  Since $\phi_{I_n}(E)\subset B(x_n, r_n)\cap E$,   we have
$$f_n(E) \subset \frac{1}{r_n}((B(x_n,r_n)\cap E)-x_n).$$
Hence by \eqref{eq: conformal1}, any limit point of $f_n(E)$ (in the Hausdorff metric) is contained in $F$ and so in $W$.

By \eqref{eq: bilipschitz0} and \eqref{eq-2}, there exists a constant $D>0$ such that for all $x,y\in V$ and $n\geq 1$,
\begin{equation}
\label{eq-3}
D^{-1}\|x-y\|\leq  \|f_n(x)-f_n(y)\|\leq D\|x-y\|.
\end{equation}
Hence the sequence $(f_n)$ is equi-continuous on $V$. Set $y_n=\phi_{I_n}^{-1}(x_n)$. Since $x_n\in \phi_{I_n}(E)$, we have $y_n\in E\subset V$.  Moreover, $f_n(y_n)=\psi_n(x_n)=0$. It follows that for every $x\in V$, $$\|f_n(x)\|=\|f_n(x)-f_n(y_n)\|\leq D\|x-y_n\|\leq D\mbox{diam}(V).$$
Hence $(f_n)$ is uniformly bounded on $V$ as well. Applying Ascoli-Arezela's theorem, we can find a uniformly convergent subsequence, say, $f_{n_k}\to f$ as $k\to \infty$.  By \eqref{eq-3}, $f$ is injective.   According to Corollaries 37.3 and 13.3 of V\"{a}is\"{a}l\"{a} \cite{Vaisala1971},  $f$ is conformal on $V$ and so is $f^{-1}$ on $f(V)$.

Since any limit point of the sequence $(f_n(E))$ is contained in $W$, we have $f(E)\subset W$ and thus $E\subset f^{-1}(f(V)\cap W)$.
Recall that a conformal map in $\R^d$ ($d\geq 2$) must be complex analytic if $d=2$ and a M\"obius transformation if $d\geq 3$ (see e.g. \cite[Theorem 4.1]{reshetnyak94}).  Hence when $d=2$, $f^{-1}(f(V)\cap W)$ is a countable union of open analytic arcs; it follows that there exists $I\in \Sigma_*$ such that $\phi_I(E)$ is contained in one piece of analytic arc, and so $E$ is contained in an analytic curve.   When $d\geq 3$, $f^{-1}(f(V)\cap W)\subset f^{-1}(W)$ so it is  contained in a $(d-1)$-dimensional hyperplane or  in a $(d-1)$-dimensional sphere. Therefore either (i) or (ii) occurs and we are done.
\end{proof}

\begin{lem}
\label{lem-4.2}
There exists $L_0>0$ such that for every $I\in \Sigma_*$,  $0<r<{\rm diam}(\phi_I(E))$ and $x\in \phi_I(E)$,
 $${\rm diam}(B(x, r)\cap \phi_I(E))\geq L_0r.$$
\end{lem}

\begin{proof} Let $I\in \Sigma_*$, $0<r<\mbox{diam}(\phi_I(E))$ and $x\in \phi_I(E)$. If $\phi_I(E)\subset B(x,r)$, then we have
${\rm diam}(B(x, r)\cap \phi_I(E))\geq {\rm diam}(\phi_I(E))> r$. In what follows we assume that  $\phi_I(E)\not\subset B(x,r)$.
 Since $x\in \phi_I(E)$, we can choose $I_1\in \Sigma_*$ such that
$$
\phi_{II_1}(E)\subset B(x,r)\mbox{ and } \phi_{I\widehat{I_1}}(E)\not\subset B(x,r).
$$
Similar to the proof of \eqref{eq-2}, we have $$C^{-2}(\mbox{diam}(E))^{-1}\leq \alpha_{II_1}/r\leq 2C(\mbox{diam}(E))^{-1},$$
 where $C$ is the constant given in the proof of Lemma \ref{lem-4.1}.
  Hence by \eqref{eq: bound diam},  $$
{\rm diam}(B(x,r)\cap \phi_I (E))\geq {\rm diam}(\phi_{II_1}(E))\geq C^{-1}\alpha_{II_1} {\rm diam}(E)\geq C^{-3}r.$$
 This completes the proof of  the lemma by letting $L_0=C^{-3}$.
\end{proof}

The next two lemmas state that if $E$ satisfies one of the conditions (i)-(ii) in Lemma \ref{lem-4.1}, there exist two subsets $E_1, E_2$ of $E$ so that $E_1+E_2$ has positive thickness.

\begin{lem}
\label{lem-4.3}
Suppose that $d=2$ and $E$ is contained in a simple non-flat analytic curve. Then there exist $I, J\in \Sigma_*$ such that $\tau(\phi_I(E)+\phi_J(E))>0$.
\end{lem}

\begin{proof}
Let $\gamma: [0,1]\to \R^2$ be a simple non-flat analytic curve which contains $E$. By analyticity, we may choose two points
$x_0, y_0\in E\cap \gamma(0,1)$ so that the slopes of the tangent lines of $\gamma$ at $x_0$ and $y_0$ are finite and  different. For convenience, we use $u$ and $v$  to denote these two slopes.

Let $0<\epsilon<|u-v|/4$. Since $\gamma$ is smooth, we can pick a small $\delta>0$ such that  the slope of every line segment connecting two different points in $B(x_0, \delta)\cap E$  lies in $(u-\epsilon, u+\epsilon)$, and the slope of every line segment connecting two different points in $B(y_0, \delta)\cap E$ lies in $(v-\epsilon, v+\epsilon)$.

Choose $I, J\in \Sigma_*$ such that $\phi_I(E)\subset B(x_0, \delta)$ and $\phi_J(E)\subset B(y_0, \delta)$. In what follows  we show that
$\tau(\phi_{I}(E)+\phi_J(E))>0$.   To see this,	let $x\in \phi_I(E)$, $y\in \phi_J(E)$ and $0<r<\min\{\mbox{diam}(\phi_I(E)), \; \mbox{diam}(\phi_J(E))\}$. Notice that
\begin{equation}\label{eq: tau(E1+E2)}B(x+y, r)\cap (\phi_I(E)+\phi_J(E))\supset (B(x, r/2)\cap \phi_I(E))+ (B(y, r/2)\cap \phi_J(E)).
\end{equation}
By Lemma \ref{lem-4.2}, there exist $x'\in B(x, r/2)\cap \phi_I(E)$ and $y'\in B(y, r/2)\cap \phi_J(E)$ such that
$$
\|x-x'\|\geq L_0r/4, \quad \|y-y'\|\geq L_0r/4.
$$
Moreover by the argument in the last paragraph, the line segment connecting $x, x'$ has slope in $(u-\epsilon, u+\epsilon)$ and that connecting $y,y'$ has slope in $(v-\epsilon, v+\epsilon)$.

Notice that the set in the right-hand side of \eqref{eq: tau(E1+E2)} contains a subset $\{x,x'\}+\{y, y'\}$ of 4 points. Hence the convex hull of $B(x+y, r)\cap (\phi_I(E)+\phi_J(E))$ contains the parallelogram with vertices in $\{x,x'\}+\{y, y'\}$. Observe that each edge of this parallelogram has length not less than $L_0r/4$, and that the angles of the  parallelogram are bounded  from below by a positive constant (for one pair of the parallel sides has slope in $(u-\epsilon, u+\epsilon)$, and the other has slope in $(v-\epsilon, v+\epsilon)$).  By elementary geometry, this parallelogram contains a ball of radius $cr$, where $c$ is a positive constant independent of $x,y$ and $r$. So the convex hull of $B(x+y, r)\cap (\phi_I(E)+\phi_J(E))$ contains a ball of radius $cr$. By definition,  $\tau(\phi_I(E)+\phi_J(E))>0$.
\end{proof}

\begin{lem}
\label{lem-4.4}
Suppose that $d\geq 3$ and $E$ is contained in a $(d-1)$-dimensional sphere of $\R^d$ but not in a hyperplane. Then there exist $I, J\in \Sigma_*$ such that $\tau(\phi_I(E)+\phi_J(E))>0$.
\end{lem}

\begin{proof}
Let $S$  be a $(d-1)$-dimensional sphere of $\R^d$ so that $S\supset E$. We first make the following.

{\bf Claim 1}. {\sl Let $F$ be a centred microset of $E$ (resp. $\phi_I(E)$ for some $I\in \Sigma_*$). Then $F$ is contained in a $(d-1)$-dimensional linear subspace  which is the tangent space (after translation to the origin) of $S$ at some $x\in E$ (resp. $x\in \phi_I(E))$.   Moreover, $F$ is not contained in a  $(d-2)$-dimensional linear subspace of $\R^d$.}

The first part of the claim simply follows from the definition of centred microsets. We leave the details to the reader.  Below we show that $F$ is not contained in any $(d-2)$-dimensional linear subspace of $\R^d$.

Suppose on the contrary that $F$ is contained in a  $(d-2)$-dimensional linear subspace, say $H$.  Then there exist $x_n\in E$, $r_n>0$, $n\geq 1$ such that $\lim_{n\to \infty} r_n$=0 and
\begin{equation*}\label{eq: microset, F}\frac{1}{r_n}((B(x_n,r_n)\cap E)-x_n)\to F \subset H
\end{equation*}
in the Hausdorff metric as $n\to \infty$. For each $n$ take $I_n\in \Sigma_*$ such that $$x_n\in \phi_{I_n}(E)\subset B(x_n, r_n)
\mbox{  and } \phi_{\widehat {I_n}}(E)\not\subset B(x_n, r_n).$$
Define $\psi_n:\; \R^d\to \R^d$ by $x\mapsto (x-x_n)/r_n$.  By a similar argument as in the proof of Lemma \ref{lem-4.1}, there exists a subsequence of  $(\psi_n\circ \phi_{I_n})$ which converges to a M{\"o}bius transformation $f$ so that $f(E)=F\subset H$.  In particular, $E\subset f^{-1}(H)$.  Since $f^{-1}$ is a M{\"o}bius transformation as well, it is of the form
\begin{equation}\label{eq: mobius psi}f^{-1}(x)=b+\frac{\alpha A(x-a)}{\|x-a\|^\epsilon},\end{equation}
where $a, b\in\R^d$, $\alpha\in\R$, $\epsilon\in\{0,2\}$ and $A$ is a $d\times d$ orthogonal matrix. Let $W'$ be a $(d-1)$-dimensional linear subspace of $\R^d$ containing $H$ and $a$.   Then $W'-a\subset W'$, hence by \eqref{eq: mobius psi} we have
\[
E\subset \left\{
\begin{array}{ll} f^{-1}(H)\subset
 f^{-1}(W')\subset AW'+b & \mbox{ if }\epsilon=0,\\
 f^{-1}(H\backslash\{a\})\subset f^{-1}(W'\backslash \{a\})\subset AW'+b & \mbox{ if }\epsilon=2.
 \end{array}
\right.
\]
However, $AW'+b$ is a hyperplane in $\R^d$. This contradicts the assumption that $E$ is not contained in a hyperplane in $\R^d$.
Hence $F$ is not contained in any $(d-2)$-dimensional linear subspace. This proves Claim 1.

Next we pick $I, J\in \Sigma_*$ so that $\phi_I(E)\cap \phi_J(E)=\emptyset$, and $\phi_I(E), \phi_J(E)$ lie on the same open semi-sphere of $S$. We claim that $\tau(\phi_I(E)+\phi_J(E))>0$.

Suppose on the contrary that  $\tau(\phi_I(E)+\phi_J(E))=0$. By Proposition \ref{prop: no mircoset in hyperlane implies thickness}, $\phi_I(E)+\phi_J(E)$ has a centred microset lying in a proper linear subspace of $\R^d$. That is, there exist $x_n\in \phi_I(E)$, $y_n\in \phi_J(E)$, $r_n>0$ with $\lim_{n\to \infty} r_n=0$ such that
\begin{equation}\label{eq: to F}
\frac{1}{2r_n}\left((B(x_n+y_n, 2r_n)\cap (\phi_I(E)+\phi_J(E)))-(x_n+y_n)\right)\to F
\end{equation}
in the Hausdorff metric, where $F$ is a compact set contained  in a $(d-1)$-dimensional linear subspace of $\R^d$, say $W$. Observe that for each $n$,
\begin{equation}\label{eq: contain 2 minisets1}
\begin{split}\frac{1}{2r_n}&\left(B(x_n+y_n,2r_n)\cap (\phi_I(E)+\phi_J(E))-(x_n+y_n)\right)\\
&\supset \frac{1}{2}\left(\frac{1}{r_n}((B(x_n,r_n)\cap \phi_I(E))-x_n) +\frac{1}{r_n}((B(y_n,r_n)\cap \phi_J(E))-y_n)\right).
\end{split}
\end{equation}
Taking a subsequence if necessary, we may assume that the sequences $\frac{1}{r_n}((B(x_n,r_n)\cap \phi_I(E))-x_n)$ and  $\frac{1}{r_n}((B(y_n,r_n)\cap \phi_J(E))-y_n)$ converge to $F_1$ and $F_2$, respectively.
By \eqref{eq: contain 2 minisets1} and \eqref{eq: to F}, $(F_1+F_2)/2\subset F\subset W$. It follows that
$\label{eq: contained in W}F_1+F_2\subset W$.
Since $0\in F_1\cap F_2$ we obtain
\begin{equation}
\label{e-F12}
F_1\subset W, \quad F_2\subset W.
\end{equation}

On the other hand by Claim 1, \begin{equation}
\label{e-F12'}
F_1\subset W_1,\quad F_2\subset W_2,
\end{equation}
where $W_1$ is the tangent space of $S$ at some point in $\phi_I(E)$,
and $W_2$ is the tangent space of $S$ at some point in $\phi_J(E)$. Since $\phi_I(E)$ and $\phi_J(E)$ are disjoint and  contained in the same open semi-sphere of $S$, $W_1\neq W_2$. It follows that either $W\cap W_1$ or $W\cap W_2$ has dimension less than $d-1$.  By \eqref{e-F12}-\eqref{e-F12'},   $F_1\subset W\cap W_1$ and $F_2\subset W\cap W_2$, hence one of $F_1$ and $F_2$ is contained in a $(d-2)$-dimensional linear subspace, which leads to a contradiction to Claim 1.
 This completes the proof of the lemma.
\end{proof}

Now we are ready to prove Theorem \ref{main thm: self-conformal}.

\begin{proof}[Proof of Theorem \ref{main thm: self-conformal}]
According to Lemmas \ref{lem-4.1}, \ref{lem-4.3} and \ref{lem-4.4}, either $\tau(E)>0$ or there exist two compact subsets $E_1, E_2$ of $E$ such that
$\tau(E_1+E_2)>0$. In either case, by Theorem \ref{thm: sum of thick sets} we see that $\oplus_nE$ has non-empty interior when $n$ is large.
\end{proof}

\section{Arithmetic sums of self-affine sets and the proof of Theorem \ref{thm: self-affine,commulative Ai}}
\label{S5}
This section is devoted to the proof of Theorem \ref{thm: self-affine,commulative Ai}. Parts (i), (ii), (iii) of the theorem will be proved separately.
\subsection{Proof of Theorem \ref{thm: self-affine,commulative Ai}(i)}
The following proposition is  a key ingredient in our proof.

\begin{prop}
\label{prop-2}
Let $\Phi=\{\phi_i(x)=Tx+a_i\}_{i=1}^\ell$ be a homogeneous affine IFS in $\R^d$. Suppose that the origin  is an interior point of ${\rm conv}(A)$, where $A=\{a_1,\ldots, a_\ell\}$. Then  there exist $\delta>0$ and $n\in \N$ such that
\begin{equation*}\label{eq:nPhi(B)'}
\oplus_n \Phi(B)\supset \oplus_nB,
\end{equation*}
where $B=B(0,\delta)$, $\Phi(B)=\bigcup_{i=1}^\ell \phi_i(B)$, and furthermore, $\oplus_nE\supset \oplus_n B$.
\end{prop}

For our purpose, below we state and prove a  generalised version of the above proposition.

\begin{prop}
\label{prop-2'}
Let $\Phi=\{\phi_i(x)=T_ix+a_i\}_{i=1}^\ell$ be an affine IFS in $\R^d$. Suppose that there exists an invertible $d\times d$ matrix $T$ and a constant $c>1$ such that
\begin{equation}
\label{e-affine1}
B(0,c^{-1})\subset T^{-k} T_{I}(B(0,1))\subset B(0,c) \quad \mbox{for all $k\in \N$ and $I\in \{1,\ldots, \ell\}^k$}.
\end{equation}
 Suppose in addition that the origin  is an interior point of ${\rm conv}(A)$, where $A=\{a_1,\ldots, a_\ell\}$. Then  there exist $\delta>0$ and $n\in \N$ such that
\begin{equation}\label{eq:nPhi(B)}
\bigoplus_{j=1}^n T_{I_j}\Phi(B)\supset \bigoplus_{j=1}^nT_{I_j}B
\end{equation}
for all $k\in \N$ and $I_1,\ldots, I_n\in \{1,\ldots, \ell\}^k$,
where $B=B(0,\delta)$,  and furthermore, $\oplus_nE\supset \oplus_n B$.
\end{prop}

\begin{proof}
Set $\rho=\min\{\|T_ix\|: \|x\|=1,\; i=1,\ldots,\ell\}$. Then $\rho>0$ and for each $1\leq i\leq \ell$,
\begin{equation}
\label{e-affine2}
B(0,\rho)\subset T_i(B(0,1))\subset B(0,1).
\end{equation}

Since $0$ is an interior point of ${\rm conv}(A)$, there exists $r>0$ so that
\begin{equation}\label{eqBallsub}
B(0,r)\subset {\rm conv}(A)\subset B(0, {\rm diam}(A)).
\end{equation}
Fix such $r$. Set $\delta=c^{-2}r/2$ and  pick $n\in\N$ such that
\begin{equation}
\label{e-affinen}
n>4c^4{\rm diam}(A)^2/(r\rho\delta).
\end{equation}  Below we show that \eqref{eq:nPhi(B)} holds for such $\delta$ and $n$.

By \eqref{e-affine1} we have
\begin{equation}
\label{e-affine3}
T^k B(0,c^{-1})\subset T_IB(0,1)\subset T^kB(0,c)
\end{equation}
for all $k\in \N$ and $I\in\{1,\ldots, \ell\}^k$. Set $B=B(0,\delta)$. By \eqref{e-affine2}, we see that
\[\Phi(B)=\bigcup_{i=1}^{\ell}(T_iB+a_i)\supset B(0,\rho \delta)+A.\]
It follows that for $k\geq 0$ and $I_1,\ldots, I_n\in\{1,\ldots,\ell\}^k$,
\begin{align}\bigoplus_{j=1}^nT_{I_j}\Phi(B)&\supset\bigoplus_{j=1}^n(T_{I_j}B(0,\rho \delta)+T_{I_j}A) \nonumber \\
&\supset\bigoplus_{j=1}^n(T^kB(0,c^{-1}\rho  \delta)+T_{I_j}A) \qquad (\text{by } \eqref{e-affine3})  \nonumber \\
&=T^kB(0,nc^{-1}\rho  \delta)+\left(\bigoplus_{j=1}^nT_{I_j}A\right). \label{eq: sum AID}\end{align}

We next show that for all  $k\geq0$ and $I\in\{1,\ldots, \ell\}^k$,
\begin{equation}\label{eq: plus A_I(D)}
T^kB(0,nc^{-1}\rho  \delta)+T_IA\supset T^kB(0,nc^{-1}\rho  \delta)+T_IB(0,  \delta).
\end{equation}
 To see this, fix $k\geq 0$ and $I\in\{1,\ldots, \ell\}^k$. By  \eqref{eqBallsub} and  \eqref{e-affine3},
\begin{equation}
\label{e-affine40}
{\rm conv}(T^{-k}T_IA)=T^{-k}T_I{\rm conv}(A)\supset T^{-k}T_IB(0,r)\supset B(0,c^{-1}r),
\end{equation}
and
\[T^{-k}T_I{\rm conv}(A)\subset T^{-k}T_IB(0,{\rm diam}(A))\subset B(0,c{\rm diam}(A)).\]
In particular,
\[{\rm diam}(T^{-k}T_IA)={\rm diam}(T^{-k}T_I{\rm conv}(A))\leq 2c{\rm diam}(A).\]
Hence by \eqref{e-affinen},
\begin{equation}\label{eqnC-1}
nc^{-1}\rho \delta> \frac{(2c{\rm diam}(A))^2}{c^{-1}r}\geq \frac{{\rm diam}(T^{-k}T_IA)^2}{c^{-1}r}.
\end{equation}
Now by  \eqref{e-affine40}-\eqref{eqnC-1}, and applying Corollary \ref{cor-4.3} (in which we replace $A$ by $T^{-k}T_IA$ and $r$ by $c^{-1}r$), we have
\[B(0,nc^{-1}\rho  \delta)+T^{-k}T_IA\supset B(0,nc^{-1}\rho  \delta)+B(0,c^{-1}r/2),\]
and thus
\begin{align*}
T^kB(0,nc^{-1}\rho  \delta)+T_I A&\supset T^kB(0,nc^{-1}\rho  \delta)+T^kB(0,c^{-1}r/2)\\
&\supset T^kB(0, nc^{-1}\rho  \delta)+ T_IB(0,c^{-2}r/2) \qquad (\text{by \eqref{e-affine3}})\\
&=T^kB(0, nc^{-1}\rho  \delta)+T_IB(0,\delta),
\end{align*}
from which \eqref{eq: plus A_I(D)} follows.

Next we apply \eqref{eq: plus A_I(D)} to prove \eqref{eq:nPhi(B)}.
 Let $k\geq 0$ and $I_1,\ldots, I_n\in\{1,\ldots, \ell\}^k$.  Write
\begin{align*}
 H_0&:=T^kB(0,nc^{-1}\rho\delta)+\left(\bigoplus_{j=1}^nT_{I_j}A\right),\\
 H_n&:=T^kB(0,nc^{-1}\rho\delta)+\left(\bigoplus_{j=1}^nT_{I_j}B(0,\delta)\right),\\
 H_m&:=T^kB(0,nc^{-1}\rho\delta)+\left(\bigoplus_{j=m+1}^nT_{I_j}A\right)+\left(\bigoplus_{j=1}^mT_{I_j}B(0,\delta)\right)
\end{align*}
for $m=1,\ldots, n-1$.  By \eqref{eq: plus A_I(D)} we have
\[T^kB(0,nc^{-1}\rho\delta)+T_{I_{m+1}}A\supset T^kB(0,nc^{-1}\rho\delta)+T_{I_{m+1}}B(0,\delta)\]
for $m\in\{0,1,\ldots, n-1\}$.
On  both sides of the above inclusion, taking sum with  $\left(\bigoplus_{j=m+2}^nT_{I_j}A\right)+\left(\bigoplus_{j=1}^mT_{I_j}B(0,\delta)\right)$  yields that
\[H_{m}\supset H_{m+1},\qquad m=0,1,\ldots, n-1.\]
Hence $H_0\supset H_1\supset \cdots\supset H_{n-1}\supset H_n$. In particular,
 $H_0\supset H_n$,  that is,
\[T^kB(0,nc^{-1}\rho\delta)+\left(\bigoplus_{j=1}^nT_{I_j}A\right)\supset T^kB(0,nc^{-1}\rho\delta)+\left(\bigoplus_{j=1}^nT_{I_j}B(0,\delta)\right),\]
which implies that
\begin{equation*}T^kB(0,nc^{-1}\rho\delta)+\left(\bigoplus_{j=1}^nT_{I_j}A\right)\supset\bigoplus_{j=1}^nT_{I_j}B(0,\delta).\end{equation*}
This combining  with \eqref{eq: sum AID} immediately yields \eqref{eq:nPhi(B)}.

 Finally we prove that for all $k\geq 0$,
\begin{equation}\label{eq: sum Phi(K)}
\oplus_n\Phi^{k+1}(B)\supset \oplus_n\Phi^k(B).
\end{equation}
To see this, fix $k\geq0$. Observe that
\begin{align*}\oplus_n\Phi^{k+1}(B)&=\bigcup_{I_1,\ldots,I_n\in\{1,\ldots,\ell\}^k}\bigoplus_{i=1}^n\phi_{I_i}(\Phi(B))\\
&=\bigcup_{I_1,\ldots, I_n\in\{1,\ldots,\ell\}^k}\bigoplus_{i=1}^n(T_{I_i}\Phi(B)+\phi_{I_i}(0))
\end{align*}
and
\begin{align*}\oplus_n\Phi^k(B)&=\bigcup_{I_1,\ldots, I_n\in\{1,\ldots,\ell\}^k}\bigoplus_{i=1}^n\phi_{I_i}(B)\\
&=\bigcup_{I_1,\ldots,I_n\in\{1,\ldots,\ell\}^k}\bigoplus_{i=1}^n(T_{I_i}B+\phi_{I_i}(0)).
\end{align*}
Meanwhile by \eqref{eq:nPhi(B)}, for all $I_1,\ldots, I_n\in\{1,\ldots, \ell\}^k$,
\[\bigoplus_{i=1}^n\left(T_{I_i}\Phi(B)+\phi_{I_i}(0)\right)\supset \bigoplus_{i=1}^n\left(T_{I_i}B+\phi_{I_i}(0)\right).\]
Hence  \eqref{eq: sum Phi(K)} holds. It follows that
\begin{equation}\label{eqphiind}
\oplus_n \Phi^{k+1}(B)\supset \oplus_n \Phi^k(B)\supset\cdots\supset \oplus_n B.
\end{equation}
Since  $\Phi^{k+1}(B)$ converges to $E$ in the Hausdorff distance as $k\to \infty$ (see e.g. \cite{Falconer2003}),   letting $k\to\infty$ in \eqref{eqphiind} yields that $\oplus_n E\supset \oplus_n B$. This completes the proof of the proposition.\end{proof}

\begin{proof}[Proof of Theorem \ref{thm: self-affine,commulative Ai}(i)]
By assumption  $E$ is not contained in a hyperplane of $\R^d$, so we can pick finitely many points in $E$, say $x_1,\ldots, x_m$ so that \begin{equation}
\label{e-conB} {\rm conv}(\{x_1, \ldots, x_m\})\supset B(z, r)
\end{equation}
 for some  $z\in\R^d$ and $r>0$. Take a large $R>0$ so that
\begin{equation}
\label{e-tt3}
\phi_i(B(0,R))\subset B(0, R), \quad i=1, \ldots, \ell.
\end{equation}
 Pick a large integer $N$ so that
 \begin{equation}
 \label{e-tt4}
 \left(\max_i\|T_i\|\right)^N\leq r/(6R).
\end{equation}
  Choose $I_1,\ldots, I_m\in \{1,\ldots, \ell\}^N$ such that $x_j\in \phi_{I_j}(E)$ for $1\leq j\leq m$. Define $W_j\in \{1,\ldots, \ell\}^{mN}$, $j=1,\ldots, m$, by
 $$W_1=I_1\cdots I_m,\; \ldots, \; W_p=I_pI_{p+1}\cdots I_mI_1\cdots I_{p-1}, \;\ldots,\; W_m=I_mI_1\cdots I_{m-1}.$$
 Since the matrices $T_i$ are commutative,  the mappings $\phi_{W_j}$, $j=1,\ldots, m$, have the same linear part.

By \eqref{e-tt3},  $E\subset B(0, R)$ and moreover for each $j$,
\begin{equation*}
\begin{split}
\phi_{W_j}(0)&\in \phi_{W_j}(B(0,R))\subset \phi_{I_j}(B(0, R)),\\
x_j&\in \phi_{I_j}(E)\subset \phi_{I_j}(B(0, R)),\\
z&\in{\rm conv}(E)\subset B(0,R).
\end{split}
\end{equation*}
It follows that $|\phi_{W_j}(0)-x_j|\leq 2\|T_{I_j}\|R$ and so
\begin{equation}
\label{e-conB'}
|\phi_{W_j}(0)+T_{W_j}z-x_j|\leq 3\|T_{I_j}\|R<r/2,
\end{equation}
where we used \eqref{e-tt4} in the last inequality.
Applying Lemma \ref{lem-3.4} (in which we take $A=\{x_1,\ldots, x_m\}$, $\delta=r/2$, $F=\{\phi_{W_j}(0)+T_{W_j}z:\; j=1,\ldots, m\}$) yields
\begin{equation}
\label{e-last1}
{\rm conv}(\{\phi_{W_j}(0)+T_{W_j}z:\; j=1,\ldots, m\})\supset U(z, r/2).
\end{equation}
(Due to \eqref{e-conB} and \eqref{e-conB'}, the conditions $B(z,r)\subset {\rm conv}(A)$ and $V_\delta(F)\supset A$ in Lemma \ref{lem-3.4} are fulfilled.)
Since  the left-hand side of \eqref{e-last1} is a compact set, we have
\[{\rm conv}(\{\phi_{W_j}(0)+T_{W_j}z:\; j=1,\ldots, m\})\supset B(z, r/2)\]
and so
\begin{equation}\label{eq: 4-7}{\rm conv}(\{\phi_{W_j}(0)+T_{W_j}z-z: j=1,\ldots, m\})\supset B(0, r/2).
\end{equation}

Let $K$ be the attractor of the IFS $\{\phi_{W_j}\}_{j=1}^m$. Then $K\subset E$. Notice that $K-z$ is the attractor of the IFS $\Psi:=\{\psi_j(x)=T_{W_j}x+\phi_{W_j}(0)+T_{W_j}z-z\}_{j=1}^m$.  To see this, it is enough to verify that $\psi_j(x-z)=\phi_{W_j}(x)-z$. Applying Proposition \ref{prop-2} to $\Psi$ and using \eqref{eq: 4-7}, we see that there exists $n\in\N$ such that $\oplus_n(K-z)$ has non-empty interior.  Since $K\subset E$, this implies  that $\oplus_nE$ has non-empty interior and we are done.
\end{proof}

\subsection{Proof of Theorem \ref{thm: self-affine,commulative Ai}(ii)}

We first introduce some notation. For $1\leq m\leq d-1$, let ${\mathcal G}_m:=G(\R^d,m)$ denote the collection of $m$-dimensional linear subspaces of $\R^d$.
It is well-known that for each $m$, ${\mathcal G}_m$ is compact endowed with the following metric
$$
\rho_m(W, W')=\|P_W-P_{W'}\|,
$$
where $P_W$ stands for the orthogonal projection onto $W$.

For any non-empty compact subset $F$ of $\R^d$, we let ${\mathcal M}_c(F)$ denote the collection of  centred microsets of $F$. For a set $H\subset \R^d$, let ${\rm span}(H)$
denote the smallest linear subspace that contains $H$. It is an elementary fact that
$$
{\rm span}(H)=\left\{\sum_{i=1}^db_ih_i:\; h_i\in H,\; b_i\in \R\right\}.
$$
Write
\begin{equation}
\label{e-affine50}
{\mathcal S}(F)=\{{\rm span}(H):\; H\in {\mathcal M}_c(F)\}.
\end{equation}
Clearly, Theorem \ref{thm: self-affine,commulative Ai}(ii) is the direct consequence of the following two propositions.
\begin{prop}
\label{prop-6.1} Let $E$ be the attractor of an affine IFS $\Phi=\{\phi_i(x)=T_ix+a_i\}_{i=1}^\ell$ on $\R^d$. Suppose that $(T_1,\ldots, T_\ell)$ is irreducible. Furthermore, assume that
for any $\epsilon>0$, there exist a non-empty compact set $F\subset E$,  an integer $m\in \{1,\ldots, d-1\}$ and $W\in {\mathcal G}_m$ such that the following property holds:
for each $V\in {\mathcal S}(F)$, there exists $W'\in {\mathcal G}_m$ so that $W'\subset V$ and $\rho_m(W', W)\leq \epsilon$.  Then $E$ is arithmetically thick.
\end{prop}

\begin{prop}
\label{prop-6.2} Let $E$ be the attractor of an affine IFS $\Phi=\{\phi_i(x)=T_ix+a_i\}_{i=1}^\ell$ on $\R^d$.  Suppose that $E$ is not contained in a hyperplane in $\R^d$. Moreover assume that  the multiplicative semigroup generated by $\{T_1,\ldots, T_\ell\}$ contains an element which has a simple dominant eigenvalue.
Then for any $\epsilon>0$, there exist a non-empty compact set $F\subset E$,  and $W\in {\mathcal G}_1$ such that the following property holds:
for each $V\in {\mathcal S}(F)$, there exists $W'\in {\mathcal G}_1$ so that $W'\subset V$ and $\rho_1(W', W)\leq \epsilon$.
\end{prop}

Below we first prove Proposition \ref{prop-6.1}.  Set $\Sigma_*=\bigcup_{n=0}^\infty \{1,\ldots, \ell\}^n$. For $I\in \Sigma_*$, let $|I|$ denote the length of $I$. We begin with an elementary fact.
\begin{lem}
\label{lem-1.2}
Let $(T_1,\ldots, T_\ell)$ be an irreducible tuple of $d\times d$ real matrices.  Let $W$ be a non-zero linear subspace of $\R^d$. Then
$$
{\rm span}\left(\bigcup_{I\in \Sigma_*:\; |I|\leq d-1} T_I(W)\right) =\R^d.
$$
\end{lem}
\begin{proof}
For $0\leq k\leq d-1$, write
$$
W_k:={\rm span}\left(\bigcup_{I\in \Sigma_*:\; |I|\leq k} T_I(W)\right).
$$
Clearly,  $W=W_0\subset W_1\subset \cdots \subset W_{d-1}$, and $W_{k+1}\supset \bigcup_{i=1}^\ell T_i(W_k)$ for  each $0\leq k\leq d-2$. Suppose on the contrary that $W_{d-1}\neq \R^d$. Since $$1\leq \dim (W_0)\leq \dim (W_1)\leq \cdots\leq \dim(W_{d-1})\leq d-1,$$
 there exists $0\leq k\leq d-2$ such that $\dim (W_{k+1})=\dim (W_{k})$ and so $W_{k+1}=W_k$. It follows that $W_k=W_{k+1}\supset \bigcup_{i=1}^\ell T_i(W_k)$, so $(T_1,\ldots, T_\ell)$ is not irreducible, leading to a contradiction.
\end{proof}

\begin{cor}
\label{cor-1.1}
Let $(T_1,\ldots, T_\ell)$ be an irreducible tuple of $d\times d$ real matrices. Then there exists $\epsilon_0>0$ such that for any $m\in \{1,\ldots, d-1\}$ and $W\in {\mathcal G}_m$,
\begin{equation}
{\rm span}\left(\bigcup_{I\in \Sigma_*:\; |I|\leq d-1} T_I(W_I)\right) =\R^d,
\end{equation}
 provided that $W_I\in {\mathcal G}_m$ and $\rho_m(W_I, W)\leq \epsilon_0$ for each $I\in \Sigma_*$ with $|I|\leq d-1$.
\end{cor}
\begin{proof}
Suppose the above conclusion is not true. Then there exists $m\in \{1,\ldots, d-1\}$ so that there are  a sequence $(\epsilon_n)$ of positive numbers with $\epsilon_n\downarrow 0$,  a sequence $(W_n)\subset {\mathcal G}_m$ and $(W_{n, I})_{n\geq 1, |I|\leq d-1}\subset  {\mathcal G}_m$ with $\rho_m( W_{n, I}, W)\leq \epsilon_n$, such that
\begin{equation*}
{\rm span}\left(\bigcup_{I\in \Sigma_*:\; |I|\leq d-1} T_I(W_{n,I})\right) \neq \R^d \mbox{ for all }n\geq 1.
\end{equation*}
Therefore there exist a sequence $(v_n)$  of unit vectors  in $\R^n$ such that
\begin{equation}
\label{e-1.1}
v_n\perp T_I(W_{n,I})\quad  \mbox{ for any } I\in \Sigma_* \mbox { with }|I|\leq d-1.
\end{equation}
Taking a subsequence if necessary we may assume that $v_n\to v$ for some unit vector $v$  and $W_n\to W$ for some $W\in {\mathcal G}_m$. Then
\eqref{e-1.1} implies that $v\perp T_I(W)$ for each $I$ with $|I|\leq d-1$. It follows that
$${\rm span}\left(\bigcup_{I\in \Sigma_*:\; |I|\leq d-1} T_I(W)\right)\subset v^{\perp}\neq \R^d,$$
leading to a contradiction with Lemma \ref{lem-1.2}.
\end{proof}

\begin{lem}
\label{lem-1.2'}
\begin{itemize}
\item[(i)]  Let $T$ be a $d\times d$ invertible real matrix.  Then for any non-empty compact $F\subset \R^d$, ${\mathcal S}(TF+a) = T {\mathcal S}(F)$ for any $a\in \R^d$.
\item[(ii)] Let $F_1,\ldots, F_k$ be non-empty compact subsets of $\R^d$. Then for each $V\in  {\mathcal S}(\bigoplus_{i=1}^k F_i )$, there exist $V_i\in {\mathcal S}(F_i)$, $i=1,\ldots, k$, such that
$$
V\supset V_1+\cdots+V_k ={\rm span} \left(\bigcup_{i=1}^k V_i\right).
$$

\end{itemize}
\end{lem}
\begin{proof}
Part (i) simply follows from a routine check, and  part (ii) follows from the property that for any $H\in {\mathcal M}_c(\bigoplus_{i=1}^k F_i)$, there exist $H_i\in {\mathcal M}_c(F_i)$, $1\leq i\leq k$, so that $H\supset \frac{1}{k}(H_1+\cdots+H_k)$.  To see this property, let  $H\in {\mathcal M}_c(\bigoplus_{i=1}^k F_i)$. By definition there exist $(x_{n,i})_{n=1}^\infty\subset F_i$, $i=1,\ldots, k$, and $r_n\downarrow 0$  such that
\begin{equation}
\label{e-affineu}
\frac{1}{r_n} \left(
\left(B(x_{n,1}+\cdots+x_{n,k}, r_n)\cap \left(\bigoplus_{i=1}^k F_i\right)\right)-(x_{n,1}+\cdots+x_{n,k})\right)\to H
\end{equation}
in the Hausdorff metric as $n\to \infty$. However, the left-hand side of \eqref{e-affineu} contains the following subset
\begin{equation}
\label{e-affineu'}
\frac{1}{k} \left(
\bigoplus_{i=1}^k
 \left[
 \frac{1}{r_n/k}
 \left(
 \left(B(x_{n,i}, {r_n}/{k}) \cap F_i\right)-x_{n,i}
  \right) \right]\right).
\end{equation}
Taking a subsequence if necessary we may assume that $\frac{1}{r_n/k}
 \left(
 \left(B(x_{n,i}, {r_n}/{k}) \cap F_i\right)-x_{n,i}
  \right)$ converges to $H_i$ for each $1\leq i\leq k$. Then $H\supset \frac{1}{k}(H_1+\cdots+H_k)$ and we are done.
\end{proof}

\begin{proof}[Proof of Proposition \ref{prop-6.1}]
Let $\epsilon_0$ be the constant given in Corollary  \ref{cor-1.1}.  By our assumption, there exist  a non-empty compact subset $F\subset E$,  an integer $m$ and $W\in {\mathcal G}_m$ such that the following property holds:
for each $V\in {\mathcal S}(F)$, there exists $W'=W'(V)\in {\mathcal G}_m$ so that $\rho_m(W', W)\leq \epsilon_0$ and $W'\subset V$.

Now we prove that $\bigoplus_{I\in \Sigma_*:\; |I|\leq d-1}\phi_I(F)$ has positive thickness. By Lemma \ref{prop: no mircoset in hyperlane implies thickness}, it is equivalent to show that
\begin{equation}
\label{e-1.3}
{\mathcal S}\left(\bigoplus_{I\in \Sigma_*:\; |I|\leq d-1}\phi_I(F)\right)=\{\R^d\}.
\end{equation}
To see this, let $V\in {\mathcal S}\left(\bigoplus_{I\in \Sigma_*:\; |I|\leq d-1}\phi_I(F)\right)$. Then by Lemma \ref{lem-1.2'}, there exists
$$(V_I)_{I\in \Sigma_*:\; |I|\leq d-1}\subset  {\mathcal S}(F)$$ such that
$$
V\supset \bigoplus_{I\in \Sigma_*:\; |I|\leq d-1} T_IV_I={\rm span} \left(\bigcup_{I\in \Sigma_*:\; |I|\leq d-1}T_IV_I\right).
$$
Recall that for each $I$, there exists $W_I\in {\mathcal G}_m$ such that $\rho_m(W_I, W)\leq \epsilon_0$ and $V_I\supset W_I$. So
$$V\supset {\rm span} \left(\bigcup_{I\in \Sigma_*:\; |I|\leq d-1}T_IW_I\right)=\R^d,$$
 where the last equality follows from Corollary \ref{cor-1.1}. This proves \eqref{e-1.3}, which implies that $\bigoplus_{I\in \Sigma_*:\; |I|\leq d-1}\phi_I(F)$ has positive thickness.
 Since $$\oplus_{\#\{I\in \Sigma_*:\; |I|\leq d-1\}} E\supset \bigoplus_{I\in \Sigma_*:\; |I|\leq d-1}\phi_I(E)\supset  \bigoplus_{I\in \Sigma_*:\; |I|\leq d-1}\phi_I(F),$$
 it follows that $E$ is arithmetically thick.
 \end{proof}

In the remaining part of this subsection, we prove Proposition \ref{prop-6.2}.  We first give the following.

\begin{lem}
\label{lem-3.2}
Let $E$ be the self-affine set generated by an affine IFS $\Phi=\{\phi_i(x)=T_ix+a_i\}_{i=1}^\ell$ on $\R^d$.
Suppose that $E$ is not contained in a hyperplane of $\R^d$. Then for any $V\in {\mathcal S}(E)$, there exists
$$
h\in \overline{\left\{\frac{T_I}{\|T_I\|}:\; I\in \Sigma_* \right\}},
$$
such that $V\supset h(\R^d)$.
\end{lem}
\begin{proof}
Let $\Gamma$ be a centred microset of $E$. Then there exist a sequence $(\epsilon_n)$ of positive numbers with $\epsilon_n\downarrow 0$, a sequence  $(x_n)$ of points in $E$ such that
$$
\frac{1}{r_n}(E\cap B(x_n, r_n)-x_n)\to \Gamma
$$
in the Hausdorff metric. For each $n$, pick $\omega_n\in \{1,\ldots, \ell\}^\N$ so that $x_n=\pi(\omega_n)$, where $\pi$ stands for the coding map for the IFS $\Phi$ (cf. \eqref{e-coding}), and
pick $k_n\in \N$ such that
\begin{equation}
\label{e-3.1}\|T_{\omega_n|k_n}\|{\rm diam}(E)<r_n\leq \|T_{\omega_n|(k_n-1)}\| {\rm diam}(E).
\end{equation}
Since  $ \|T_{\omega_n|(k_n-1)}\|\leq \|T_{\omega_n|k_n}\|\cdot \max_{1\leq i\leq \ell}\|T_i^{-1}\|$, the above inequality implies that
$$
\frac{\|T_{\omega_n|k_n}\|}{r_n}\in [\gamma_1, \gamma_2),
$$
where
$$\gamma_1:=\left({\rm diam}(E)\max_{1\leq i\leq \ell}\|T_i^{-1}\|\right)^{-1},\quad  \gamma_2:=({\rm diam}(E))^{-1}. $$
By \eqref{e-3.1}, we have
$E\cap B(x_n, r_n)\supset \phi_{\omega_n|k_n}(E)$, so
$$
(E\cap B(x_n, r_n))-x_n\supset \phi_{\omega_n|k_n}(E)-\phi_{\omega|k_n}(\pi \sigma^{k_n}\omega_n)=T_{\omega_n|k_n}(E-\pi \sigma^{k_n}\omega_n),
$$
where $\sigma$ is the left-shift map on $\{1,\ldots, \ell\}^\N$. It follows that
$$
\frac{1}{r_n}(E\cap B(x_n, r_n)-x_n)\supset \frac{\|T_{\omega_n|k_n}\|}{r_n}\cdot \frac{T_{\omega_n|k_n}}{\|T_{\omega_n|k_n}\|}(E-\pi\sigma^{k_n}\omega_n).
$$
Taking a subsequence if necessary, we may assume that
$$
\frac{\|T_{\omega_n|k_n}\|}{r_n}\to c\in [\gamma_1,\gamma_2],\quad  \frac{T_{\omega_n|k_n}}{\|T_{\omega_n|k_n}\|}\to h,\quad \pi\sigma^{k_n}\omega_n\to z.
$$
Then we have $\Gamma\supset ch(E-z)$. It follows that ${\rm span}(\Gamma)\supset  h({\rm span}(E-z))=h(\R^d)$, where in the last equality we use the assumption that $E$ is not contained in a hyperplane.
\end{proof}

\begin{proof}[Proof of Proposition \ref{prop-6.2}]

First choose a large $R>0$ such that $\phi_i(B_R)\subset B_R$ for all $1\leq i\leq \ell$, where $B_R:=B(0, R)$. Since $E$ is not in a hyperplane, we can pick
points $z_1,\ldots, z_{d+1}$ so that ${\rm conv}(\{z_1,\ldots, z_{d+1}\})$ has non-empty interior. Hence there exists $\delta>0$ such that
${\rm conv}(\{z_1',\ldots, z_{d+1}'\})$ has non-empty interior for any tuple $(z_1',\ldots, z_{d+1}')$ of points with $|z_i'-z_i|<\delta$ for all $i$. Pick $I_1,\ldots, I_{d+1}\in \Sigma_*$ such that $\phi_{I_i}(B_R)\subset B(z_i, \delta)$ for $1\leq i\leq d+1$.

Pick $W\in \Sigma_*$ so that $\lambda$  is a simple eigenvalue of $T_W$ and $|\lambda|$ is greater than the magnitude of any other eigenvalue of $T_W$.
Replacing $W$ by $W^2$ if necessary, we may assume that $\lambda>0$.  Choosing a suitable basis of $\R^d$  if necessary, we may assume that $T_W$ is in its real Jordan canonical form so that
$T_W(e_1)=\lambda e_1$, where $e_1=(1,0,\ldots, 0)$. Then
\begin{equation}
\label{e-affine-appro}
\lambda^{-n}T_W^n\to {\rm diag}(1,0,\ldots,0)\quad \mbox{ as } n\to \infty.
\end{equation}
Define
$$
K:=\{(x_1,\ldots, x_d)\in \R^d:\; x_1\geq 0,\; x_1^2+\cdots+x_d^2\leq 2 x_1^2\}.
$$
Then $K$ is a cone in $\R^d$. By \eqref{e-affine-appro}  there exists a large integer $N$ so that
$T_W^N(K\backslash \{0\})\subset {\rm interior}(K)$.

Since $(T_1,\ldots, T_\ell)$ is irreducible, for each $1\leq i\leq d+1$,  there exists $J_i\in \Sigma_*$ so that $|J_i|\leq d-1$ and
\begin{equation}
\label{e-3.2}
t_i:=e_1 T_{I_iJ_i}e_1^*\neq 0,
\end{equation}
 where $e_1^*$ denotes the transpose of $e_1$.  (To see the existence, simply notice that
$${\rm span}\left(\bigcup_{J\in \Sigma_*:\; |J|\leq d-1} T_J e_1^* \right)=\R^d$$ by Lemma \ref{lem-1.2}.)

Fix the above $J_1,\ldots, J_{d+1}$. Pick a large $k$ so that
\begin{equation}
\label{e-affinecone}
(T_W^{Nk} T_{I_iJ_i} T_W^{Nk} )^2 (K\backslash \{0\}) \subset {\rm interior}(K).
\end{equation}
(To see the existence of $k$, notice that the diagonal matrix $M={\rm diag}(1,0,\ldots,0)$ satisfies the cone condition $M (K\backslash \{0\}) \subset {\rm interior}(K)$, so there exists $\epsilon>0$ such that if $M'$ is $\epsilon$-close to $M$, then $M'$ also satisfies the cone condition that
$M' (K\backslash \{0\}) \subset {\rm interior}(K)$. Now for given $i$, by \eqref{e-affine-appro}-\eqref{e-3.2} it is easily checked that
$$
\frac{T_W^{Nk} T_{I_iJ_i} T_W^{Nk}}{\lambda^{2Nk}\cdot t_i}\to M
$$
as $k\to \infty$.  As $t_i$ might be negative, so we take square of $T_W^{Nk} T_{I_iJ_i} T_W^{Nk}$ in \eqref{e-affinecone}.)

 Now set $$\Psi=\{ \phi^2_{W^{Nk}I_iJ_iW^{Nk}}\}_{i=1}^{d+1},$$
and let $H$ be the attractor of $\Psi$. Clearly $H\subset E\subset B_R$.
By the aforementioned analysis, the linear parts of the mappings in $\Psi$ satisfy the cone condition \eqref{e-affinecone}, and moreover, for any given $y_i\in H$,
$i=1,\ldots, d+1$,  we have $y_i\in B_R$ and therefore
$$
\phi_{I_iJ_iW^{Nk}W^{Nk}I_iJ_iW^{Nk}}(y_i)\in B(z_i,\delta),
$$
so the set $\{\phi_{I_iJ_iW^{Nk}W^{Nk}I_iJ_iW^{Nk}}(y_i)\}_{i=1}^{d+1}$ is not contained in a hyperplane. It implies that
$$\{\phi^2_{W^{Nk}I_iJ_iW^{Nk}}(y_i)\}_{i=1}^{d+1}$$
is not contained in a hyperplane. Hence $H$  is not contained in a hyperplane.

For convenience, rewrite $\Psi$ as $\{\psi_i(x)=T_i'x+a_i'\}_{i=1}^{d+1}$.  By \eqref{e-affinecone},  $T_i'(K\backslash \{0\})\subset {\rm interior}(K)$ for any $i$.  It follows that for each element
\begin{equation}
\label{e-affineH}
h\in \Lambda:=\overline{\left\{\frac{T_I'}{\|T_I'\|}:\; I\in \bigcup_{n\geq 0} \{1,\ldots, d+1\}^n\right\}},
\end{equation}
$h(K)\subset K$.  Since ${\rm interior}(K)\neq \emptyset$ and $h\neq 0$, we have $h(K)\neq \{0\}$.   It implies that $h(\R^d)\cap K\supset h(K)\neq \{0\}$.

Let $\epsilon>0$. Since $T_1'(K\backslash \{0\})\subset {\rm interior}(K)$,  by the generalised Perron-Frobenius theorem (see e.g. \cite[Theorem B.1.1]{LemmensNussbaum2012}), $T_1'$ has a unit eigenvector $v\in K$, and moreover, there exists $n\in \N$ such that every unit vector $v'\in (T_1')^nK$ is $\epsilon$-close to $v$.

Fix the above $n$. Applying Lemma \ref{lem-3.2} to the IFS $\Psi$,  we see that for any $V\in {\mathcal S}(\psi_1^n(H))=(T_1')^n {\mathcal S}(H)$, $$V\supset (T_1')^n h(\R^d)
\supset (T_1')^n (h(\R^d)\cap K)$$
 for some $h\in \Lambda$, where $\Lambda$ is defined as in \eqref{e-affineH}. Since $h(\R^d)\cap K\neq \emptyset$,  $V$ contains a unit vector which is $\epsilon$-close to $v$. Therefore the conclusion of the proposition holds for  $F:=\psi_1^n(H)$.
\end{proof}

\subsection{Proof of Theorem \ref{thm: self-affine,commulative Ai}(iii)}
In this subsection, let $E$ be the attractor of an affine IFS $\{\phi_i(x)=T_ix+a_i\}_{i=1}^\ell$ in $\R^2$ and assume that $E$ is not contained in a straight line.
Theorem \ref{thm: self-affine,commulative Ai}(iii) states that $E$ is arithmetically thick.  Below we prove this statement.

First we give two elementary lemmas.

\begin{lem}
\label{lem-6.9}
Let $T=\left(\begin{array}{ll}
c & e\\
0 &d
\end{array}
\right)$, where $d>c>0$ and $e\in \R$. Let $\epsilon>0$ so that $\epsilon|e|< d-c$.  Define  a cone $K\subset \R^2$ by $K=\{(x,y)\in \R^2:\; y\geq \epsilon |x|\}$. Then
$T(K\backslash \{0\})\subset {\rm interior}(K)$.
\end{lem}

\begin{proof}
Let $(x,y)\in K\backslash\{0\}$. Then $T(x,y)=(cx+ey, dy)$.  Clearly, $$
\epsilon |cx+ey|\leq c\epsilon |x|+\epsilon |e| y\leq (c+\epsilon |e|)y<dy.$$
So $T(x,y)\in  {\rm interior}(K)$.
\end{proof}

 \begin{lem}
 \label{lem-6.10}
Let $T_i=\left(\begin{array}{ll}
c & e_i\\
0 &d
\end{array}
\right)$, $i=1,\ldots, \ell$, where $c>d>0$ and $e_i\in \R$. Set $T=\left(\begin{array}{ll}
c & 0\\
0 &d
\end{array}
\right)$. Then there exists a constant $\lambda>1$ such that for any $n\geq 0$ and  $I\in \{1,\ldots, \ell\}^n$,
$$
B(0,\lambda^{-1})\subset T^{-n}T_IB(0,1)\subset B(0,\lambda).
$$
 \end{lem}
 \begin{proof}
  It is readily checked that  for $I=i_1\ldots i_n$,
 $$
 T^{-n}T_I=\left(\begin{array}{cc}
 1 & \sum_{k=1}^n c^{-1} (d/c)^{n-k} e_{i_k}\\
 0 & 1
 \end{array}
 \right),
 $$
 and so
 $$
 (T^{-n}T_I)^{-1}=\left(\begin{array}{cc}
 1 & -\sum_{k=1}^n c^{-1} (d/c)^{n-k} e_{i_k}\\
 0 & 1
 \end{array}
 \right).
 $$

Since $c>d>0$, $|\sum_{k=1}^n c^{-1} (d/c)^{n-k} e_{i_k}|$ is bounded above by a constant, say $u$.  It follows that
$\|T^{-n}T_I\|\leq 1+u$ and $\|(T^{-n}T_I)^{-1}\|\leq 1/(1+u)$. Now the conclusion of the lemma follows by letting $\lambda=1+u$.
 \end{proof}

\begin{proof}[Proof of Theorem \ref{thm: self-affine,commulative Ai}(iii)] We consider separately the two different cases: (1) $(T_1,\ldots, T_\ell)$ is irreducible; (2) $(T_1,\ldots, T_\ell)$ is reducible.

First assume that the tuple $(T_1,\ldots, T_\ell)$ is irreducible. Set $T_i'=|{\rm det}(T_i)|^{-1/2}T_i$, $i=1,\ldots, \ell$. Then ${\rm det}(T_i')=\pm 1$ for all $1\leq i\leq \ell$. Let $H$ denote the multiplicative semigroup generated by $\{T_1',\ldots, T_\ell'\}$. It is clear that either $\rho(A)=1$ for all $A\in H$, where $\rho(\cdot)$ denotes the spectral radius,  or there exists $A\in H$ so that $\rho(A)>1$.  It is known (see \cite[Theorem 2]{ProtasovVoynov2017}) that  the first scenario occurs if and only if there exists an invertible matrix $J$ such that $J^{-1}AJ$ is orthogonal for all $A\in H$.   Hence if the first scenario occurs, then $J^{-1}\circ \phi_i\circ J$ is a similarity map for each $1\leq i\leq \ell$, so $J^{-1}(E)$ (which is the attractor of the IFS $\{J^{-1}\circ \phi_i\circ J\}_{i=1}^\ell$) is a self-similar set; by Corollary \ref{cor-similar}, $J^{-1}(E)$ is arithmetically thick, and so is $E$.   Now suppose that the second scenario occurs, i.e.~there exists $A\in H$ so that $\rho(A)>1$. But since ${\rm det}(A)=\pm 1$, $\rho(A)>1$ means that $A$ has a simple dominant eigenvalue.    Hence in such case, the semigroup generated by $\{T_1,\ldots, T_\ell\}$ also contains an element which has a simple dominant eigenvalue; so by Theorem \ref{thm: self-affine,commulative Ai}(ii), $E$ is arithmetically thick.

In what follows, we assume that  $(T_1,\ldots, T_\ell)$ is reducible. Then in a suitable basis of $\R^2$, $T_1,\ldots, T_\ell$ are upper triangular matrices, say,
$$T_i=\left(\begin{array}{ll}
c_i & e_i\\
0 & d_i
\end{array}
\right),\quad i=1,\ldots, \ell.
$$
Below we show that $E$  is arithmetically thick.

Pick a large $R>0$ so that $\phi_i(B_R)\subset B_R$, where $B_R=B(0, R)$. Then $E\subset B_R$. Since $E$ is not contained in a straight line,
replacing $\Phi$ by a sub-IFS of $\Phi^n$ for some large $n$,  we may assume that \begin{itemize}
\item[({\bf A}1)] $\phi_i(B_R)$, $i=1,\ldots, \ell$, are disjoint;  and
\item[({\bf A}2)]  there exist $z\in \R^2$ and $r>0$ such that for any $y_i\in \phi_i(B_R)$, $i=1,\ldots, \ell$, ${\rm conv}(\{y_1,\ldots, y_\ell\})\supset B(z,r)$.
\end{itemize}
Furthermore, replacing $\phi_i$ by $\phi_i^2$ if necessary, we may assume that
$$c_i>0,\; d_i>0 \;\mbox { for all }i=1,\ldots, \ell.
$$

Below we will consider  3 possible cases: (a) $c_i=d_i$ for all $1\leq i\leq \ell$; (b) there exists $i\in \{1,\ldots, \ell\}$ so that $c_i>d_i$; (c)  there exists $i\in \{1,\ldots, \ell\}$ so that $c_i<d_i$.

If Case (a) occurs, then it is readily checked that $T_iT_j=T_jT_i$ for all $i,j$,  so by Theorem \ref{thm: self-affine,commulative Ai}(i), $E$ is arithmetically thick.

Next assume that Case (b) occurs, i.e.~there exists $i\in \{1,\ldots, \ell\}$ so that $c_i>d_i$. Without loss of generality,  assume that $c_1>d_1$. For $k\in\N$, let $1^k$ denote the word in $\{1,\ldots, \ell\}^k$ consisting of $k$ many $1$'s. Then we can pick a large $k$ so that
\begin{equation}
\label{e-affine55}
c_1\cdots c_{\ell}c_1^k>d_1\cdots d_{\ell}d_1^k.
\end{equation}
Define   $W_1=12\ldots\ell 1^k$, $W_2=23\ldots\ell11^k$, $\ldots$,  $W_{\ell}=\ell12\ldots(\ell-1)1^k$. Since $T_1, \ldots, T_{\ell}$ are upper triangular matrices, it is easily seen that  $T_{W_1}, \ldots, T_{W_{\ell}}$ are upper triangular  with a common diagonal part ${\rm diag}(c_1\cdots c_{\ell}c_1^k, d_1\cdots d_{\ell}d_1^k)$.  Let $F$ be the attractor of $\{\phi_{W_i}\}_{i=1}^{\ell}$.  Clearly $F\subset E$. The assumptions ({\bf A}1)-({\bf A}2) imply  that $F\subset B_R$ and
 \begin{equation}\label{convx1}
 {\rm conv}(\{y_1,\ldots, y_{\ell}\})\supset B(z,r)
 \end{equation}
 for any $y_i\in\phi_{W_i}(B_R)$, $i=1,\ldots, \ell$.
It follows that  $F$ is not contained in a straight line, and  $z\in B(0,R)$. Again by \eqref{convx1} we have
$${\rm conv}\{\phi_{W_i}(z)\}_{i=1}^{\ell}\supset B(z,r),$$
and so
\begin{equation}\label{eqconvphiWj}
{\rm conv}\{\phi_{W_i}(z)-z\}_{i=1}^{\ell}\supset B(0,r).
\end{equation}
It is easy to check  that $F-z$ is the attractor of the IFS
\begin{align*}
\left\{T_{W_i}x+\phi_{W_i}(z)-z\right\}_{i=1}^{\ell}.
\end{align*}
Set $T= {\rm diag}(c_1\cdots c_{\ell}c_1^k, d_1\cdots d_{\ell}d_1^k)$.  By  \eqref{e-affine55} and Lemma \ref{lem-6.10},  there exists a constant $\lambda>1$ so that
$$
B(0,\lambda^{-1})\subset
T^{-n}T_{W_{i_1}\cdots W_{i_n}} B(0,1)\subset B(0,\lambda)
$$
for any $n\geq 0$ and $i_1,\ldots, i_n\in \{1,\ldots, \ell\}$.
 Now  applying Proposition \ref{prop-2'} to the IFS $\left\{ T_{W_i}x+\phi_{W_i}(z)-z\right\}_{i=1}^{\ell}$, we see that  $F-z$ is arithmetically thick, and so is $E$.

 Finally assume that Case (c) occurs, i.e.~there exists $i\in \{1,\ldots, \ell\}$ so that $c_i<d_i$. Without loss of generality,  assume that $c_1<d_1$.
 Pick a large $k$ so that
 \begin{equation}
 \label{e-affine32}
 c_1\cdots c_{\ell}c_1^k<d_1\cdots d_{\ell}d_1^k.
 \end{equation}
 Define $W_1,\ldots, W_\ell$ as in the previous argument for Case (b).  Then $T_{W_1}, \ldots, T_{W_{\ell}}$ are upper triangular  with a common diagonal part ${\rm diag}(c_1\cdots c_{\ell}c_1^k, d_1\cdots d_{\ell}d_1^k)$. Let $F$ be the attractor of $\{\phi_{W_i}\}_{i=1}^{\ell}$.  Similarly,  $F\subset E$ and $F$ is not contained in a straight line. If all the matrices $T_{W_i}$ are the same, then by Theorem \ref{thm: self-affine,commulative Ai}(i), $F$ is arithmetically thick and so is $E$. Below we assume that at least two of the matrices  $T_{W_i}$ are different, say   $T_{W_1}\neq T_{W_2}$.

 By \eqref{e-affine32} and Lemma \ref{lem-6.9},
 there exists a small $\epsilon>0$ such that
 \begin{equation}
 \label{e-affine33}
 T_{W_i}(K\backslash \{0\})\subset {\rm interior}(K)
 \end{equation}
  for all $1\leq i\leq \ell$, where
 $
 K:=\{(x,y)\in \R^2:\; y\geq \epsilon |x|\}.
 $
 It follows that for each element
\begin{equation}
\label{e-affineH'}
h\in \Lambda:=\overline{
\left\{
\frac{T_{ W_{i_1}\cdots W_{i_n}}} {\|T_{W_{i_1}\cdots W_{i_n}}\|}:\; n\geq 1, \; i_1\ldots i_n\in \{1,\ldots, \ell\}^n \right\}
},
\end{equation}
we have
$h(K)\subset K$.  Since ${\rm interior}(K)\neq \emptyset$ and $h\neq 0$, we have $h(K)\neq \{0\}$.   It implies that $h(\R^d)\cap K\supset h(K)\neq \{0\}$. Hence by Lemma \ref{lem-3.2}, for any $V\in {\mathcal S}(F)$, $V$ contains a non-zero vector in $K$.

 By \eqref{e-affine33} and the generalised Perron-Frobenius theorem (see e.g. \cite[Theorem B.1.1]{LemmensNussbaum2012}), for each $1\leq i \leq \ell$, the
 matrix $T_{W_i}$ has an eigenvector $v_i$ corresponding to  the eigenvalue $d_1\ldots d_\ell d_1^k$ so that $\|v_i\|=1$ and $v_i\in K$, and moreover for any $v\in K$, \begin{equation}
 \label{e-affine34}
 \frac{T_{W_i}^nv}{\|T_{W_i}^nv\|}\to v_i \quad \mbox{ as }n\to \infty.
 \end{equation}
 Since $T_{W_1}\neq T_{W_2}$, it is readily checked that $v_1\neq v_2$ (and moreover, $v_1$ and $v_2$ are linearly independent). Pick a small enough $\delta>0$ so that if $v_1'$ is $\delta$-close to $v_1$, and $v_2'$ is $\delta$-close to $v_2$, then $v_1'$ and $v_2'$ is linearly independent.
 By \eqref{e-affine34}, there exists a large $n$ such that any unit vector in $T_{W_i}^n(K)$ is $\delta$-close to $v_i$, $i=1,2$.  Fix such $n$. We claim that
 $\phi_{W_1}^n(F)+\phi^n_{W_2}(F)$ has positive thickness.  To prove this, by Lemma \ref{prop: no mircoset in hyperlane implies thickness} it is equivalent to show that ${\mathcal S} (\phi_{W_1}^n(F)+\phi^n_{W_2}(F))=\{\R^2\}$. To see it, let $V\in  {\mathcal S} (\phi_{W_1}^n(F)+\phi^n_{W_2}(F))$. Then by Lemma \ref{lem-1.2'}, $V\supset T_{W_1}^n V_1+ T_{W_2}^n V_2$ for some $V_1, V_2\in {\mathcal S}(F)$. Since both $V_1$ and $V_2$ contain non-zero vectors in $K$,
 we see that $T_{W_1}^n V_1$ contains a unit vector which is $\delta$-close to $v_1$, and $T_{W_2}^n V_2$ contains a unit vector which is $\delta$-close to $v_2$.  Hence $V$ contains two linearly independent vectors and so $V=\R^2$, which proves the claim. Since $$\phi_{W_1}^n(F)+\phi^n_{W_2}(F)\subset F+F\subset E+E,$$
  it follows  from Theorem \ref{thm: sum of thick sets} that $E$ is arithmetically thick.
 This completes the proof of Theorem \ref{thm: self-affine,commulative Ai}(iii).
\end{proof}
\section{A result on the arithmetic sums of rotation-free self-similar sets}
\label{S8}

In this section, we prove the following result on the arithmetic sums of rotation-free self-similar sets in $\R^d$, which partially generalises \cite[Theorem 7]{NikodemPales2010}.
\begin{thm}\label{thm: sum of E=sum of conv(E)}
Let $\{\phi_i(x)=\rho_ix+a_i\}_{i=1}^\ell$ be an IFS in $\R^d$ with attractor $E$, where $0<\rho_i<1$ and $a_i\in\R^d$ for $1\leq i\leq\ell$.  Let $F$ be the set of the fixed points of $\phi_i$'s. Then for every $n\geq 1+\ell/(\min_i\rho_i)$, $\oplus_nE=n\;{\rm conv}(F)$.
\end{thm}

In \cite{NikodemPales2010} Nikodem and P\'{a}les proved a general result on the arithmetic sums of fractal sets in Banach spaces which, applied to Euclidean spaces, yields that if $E$ is the attractor of a homogeneous IFS $\{\rho x+a_i\}_{i=1}^\ell$ in $\R^d$,    then there exists $n$ so that $\oplus_nE=n\;{\rm conv}(F)$.

\begin{proof}[Proof of Theorem \ref{thm: sum of E=sum of conv(E)}]
First we show that $\phi_i({\rm conv}(F))\subset {\rm conv}(F)$ for any  $1\leq i\leq \ell$. To see this, let $i\in\{1,\ldots,\ell\}$. Let $b_j$ be the fixed point of $\phi_j$, then $b_j=a_j/(1-\rho_j)$ for $1\leq j\leq \ell$. For any probability vector $(p_1,\ldots, p_{\ell})$,
\begin{align*}
\phi_i(p_1b_1+\cdots+p_{\ell}b_{\ell})&=\rho_i(p_1b_1+\cdots+p_{\ell}b_{\ell})+(1-\rho_i)b_i\\
&=(1-\rho_i+\rho_ip_i)b_i+\sum_{1\leq j\leq \ell, \ j\neq i}\rho_ip_jb_j\\&\in {\rm conv}(F).
\end{align*}
Hence $\phi_i({\rm conv}(F))\subset {\rm conv}(F)$, as was to be shown. Since ${\rm conv}(F)$ is compact, it follows that $E\subset {\rm conv}(F)$.

Let $\Sigma_*$ denote the collection of  finite words over the alphabet $\{1,\ldots, \ell\}$, including the empty word $\varepsilon$.  Set $\phi_{\varepsilon}=id$, the identity map of $\R^d$.  For $I\in \Sigma_*$ let $|I|$ denote the length of $I$.

Write  $\rho_{\min}=\min_i\rho_i$ and fix an integer $n\geq 1+\ell/\rho_{\min}$.  To prove the theorem, we first construct recursively a sequence $\{(\Omega_{k,1}, \ldots, \Omega_{k, n})\}_{k\geq 1}$ of $n$-tuples of subsets of $\Sigma_*$.  We start by setting $\Omega_{1,1}=\cdots=\Omega_{1,n}=\{\varepsilon\}$.  Suppose we have defined well the tuple $(\Omega_{k,1}, \ldots, \Omega_{k, n})$ for some $k$.  Choose one word $I_k$ from $\bigcup_{i=1}^n\Omega_{k,i}$ so that $$\rho_{I_k}=\max\left\{\rho_J:\; J\in\bigcup_{i=1}^n\Omega_{k,i}\right\}.$$
 Then choose one index $j_k\in \{1,\ldots, n\}$ so that $I_k\in \Omega_{k,j_k}$, and  define  $(\Omega_{k+1,1}, \ldots, \Omega_{k+1, n})$ by
 \begin{equation}
 \label{e-13}
 \Omega_{k+1,j_k}=\left(\Omega_{k,j_k}\backslash \{I_k\}\right)\cup \{I_ki:\; i=1,\ldots, \ell\}
\end{equation}
 and
 \begin{equation}
 \label{e-14}
  \Omega_{k+1,i}= \Omega_{k,i} \quad \mbox{ for all }i\neq j_k.
 \end{equation}
  Continuing the above process, we define well the whole sequence $\{(\Omega_{k,1}, \ldots, \Omega_{k, n})\}_{k\geq 1}$.

  By the above construction, it is readily checked that
  \begin{equation*}
  \label{e-11}
  \min\left\{\rho_J:\; J\in\bigcup_{i=1}^n\Omega_{k,i}\right\}\geq \rho_{\min} \cdot \max\left\{\rho_J:\; J\in\bigcup_{i=1}^n\Omega_{k,i}\right\} \mbox{ for each }k\in \N
  \end{equation*}
and
\begin{equation}
\label{e-tt5}
\inf\left\{|J|:\; J\in \bigcup_{i=1}^n \Omega_{k, i}\right\}\to \infty \;\mbox{ as } k\to \infty.
\end{equation}

 Next we claim that for any $k\in \N$,
 \begin{equation}
 \label{e-12}
 \bigoplus_{i=1}^n \bigcup_{I\in \Omega_{k+1,i}}\phi_I({\rm conv}(F))= \bigoplus_{i=1}^n \bigcup_{I\in \Omega_{k,i}}\phi_I({\rm conv}(F)).
 \end{equation}
  By \eqref{e-13}-\eqref{e-14}, to prove \eqref{e-12} it suffices to show that
  \begin{equation}
  \label{e-15}
  H_k+\left(\bigcup_{i=1}^\ell\phi_{I_ki}({\rm conv}(F))\right)=H_k+\phi_{I_k}({\rm conv}(F)),
  \end{equation}
  where $H_k:= \bigoplus_{1\leq i\leq n,\; i\neq j_k} \bigcup_{J\in \Omega_{k,i}}\phi_J({\rm conv}(F))$.
  Since $\bigcup_{i=1}^\ell \phi_i({\rm conv}(F))\subset {\rm conv}(F)$, the direction ``$\subset$'' in \eqref{e-15} is obvious. We only need to prove the other direction.

 Notice that  $\bigcup_{i=1}^\ell \phi_i({\rm conv}(F))\supset \bigcup_{i=1}^\ell \phi_i(F)\supset F$ (since  $F$ consists of the fixed points of $\phi_i$'s).  Hence to prove the direction ``$\supset$'' in \eqref{e-15},
 it is enough to show that $H_k+\phi_{I_k}(F)\supset H_k+\phi_{I_k}({\rm conv}(F))$, or equivalently, to show that
 \begin{equation}
  \label{e-16}
  H_k+\rho_{I_k} F\supset H_k+\rho_{I_k}{\rm conv}(F).
  \end{equation}

 According to the definition of $H_k$, we can write $H_k$ as a union of finitely many homothetic copies of
 ${\rm conv}(F)$, say $r_u {\rm conv}(F)+b_u$ ($u=1, 2,\ldots$),  with $r_u\geq (n-1) \rho_{I_k}\rho_{\min}\geq \ell \rho_{I_k}$ and $b_u\in \R^d$.  By Lemma \ref{lem-3.1} (in which we take $A=F$ and $\epsilon=\rho_{I_k}/r_u$) we have
 \[{\rm conv}(F)+(\rho_{I_k}/r_u)\cdot F={\rm conv}(F)+(\rho_{I_k}/r_u)\cdot {\rm conv}(F),\]
and  thus
 $$(r_u {\rm conv}(F)+b_u)+\rho_{I_k}F=(r_u {\rm conv}(F)+b_u)+\rho_{I_k}{\rm conv}(F)$$ for each $u$. Taking union over $u$ yields  $H_k+\rho_{I_k} F= H_k+\rho_{I_k}{\rm conv}(F)$. Hence \eqref{e-16} holds, and thus \eqref{e-12} holds.

 Applying   \eqref{e-12} repeatedly, we see that for each $k$,
 \begin{equation}
 \label{e-tt6}
 \bigoplus_{i=1}^n \bigcup_{I\in \Omega_{k,i}}\phi_I({\rm conv}(F))=\bigoplus_{i=1}^n \bigcup_{I\in \Omega_{k-1,i}}\phi_I({\rm conv}(F))=\cdots=\oplus_n {\rm conv}(F).
 \end{equation}

 Now for given $k\in \N$, by \eqref{e-tt5} there exists a large integer $k'$  so that
 $$
 \inf\left\{|J|:\; J\in \bigcup_{i=1}^n \Omega_{k', i}\right\}\geq k.
 $$
 Since $\phi_i({\rm conv}(F))\subset {\rm conv}(F)$ for each $i$,  the above inequality implies that
 $$\bigcup_{I\in \Omega_{k',i}}\phi_I({\rm conv}(F))\subset \bigcup_{I\in \Sigma_k}\phi_I({\rm conv}(F)), \quad i=1,\ldots, n.
 $$
Hence by \eqref{e-tt6},
 $$\oplus_n \bigcup_{I\in \Sigma_k}\phi_I({\rm conv}(F))\supset \bigoplus_{i=1}^n  \bigcup_{I\in \Omega_{k',i}}\phi_I({\rm conv}(F)) =\oplus_n {\rm conv}(F).$$
    Letting $k\to \infty$, we obtain $\oplus_n E\supset \oplus_n {\rm conv}(F)$.  Since $E\subset {\rm conv}(F)$, we get
 $$\oplus_n E= \oplus_n {\rm conv}(F)=n{\rm conv}(F)$$
 and we are done.
 \end{proof}
\begin{rem}
 The reader may check that under the assumption of Theorem \ref{thm: sum of E=sum of conv(E)}, one has ${\rm conv}(F)={\rm conv}(E)$, therefore
 $\oplus_n E=n {\rm conv}(F)=n{\rm conv}(E)$ for  large enough $n$.
  Below we give an example to show this property may  fail in the rotation case.
\end{rem}

\begin{ex}\label{ex: example to show nE not equal nconv(E)}Let $\phi_1,\phi_2$ be the homotheties in $\R^2$ with ratio $\frac{1}{4}$ and fixed points $(1,0), (0,1)$ respectively. Let $\phi_3(x)=\frac{1}{4}R_{-\frac{\pi}{2}}(x-(1,0))$, where $R_{-\frac{\pi}{2}}$ denotes the rotation matrix in $\R^2$ with angle $-\frac{\pi}{2}$. Let $E$ be the attractor of $\{\phi_i\}_{i=1}^3$. Let $T$ be the triangle with vertices $(0,0), (1,0)$ and $(0,1)$. Below we show that ${\rm conv}(E)=T$ but $\oplus_nE\neq nT$ for all $n\in\N$.
\end{ex}

\begin{proof} Since $(1,0)$ and $(0,1)$ are the fixed points of  $\phi_1$ and $\phi_2$ respectively, we have $(1,0), (0,1)\in E$, and so $(0,0)=\phi_3((1,0))\in E$. Hence $$T={\rm conv}(\{(0,0), (1,0), (0,1)\})\subset {\rm conv}(E).$$ On the other hand, it is direct to  check that  $\phi_i(T)\subset T$ for $i=1,2,3$, see Figure \ref{fig2}.  This implies $E\subset T$ and hence ${\rm conv}(E)\subset T$.

\begin{figure}
\centering
\begin{tikzpicture}
\draw[->] (0,0)--(0,5);
\draw[->] (0,0)--(5,0);
\draw (4,0)--(0,4);
\draw [fill=gray] (0,4)--(0,3)--(1,3)--(0,4);
\draw [fill=gray] (0,0)--(1,1)--(0,1)--(0,0);
\draw [fill=gray] (4,0)--(3,1)--(3,0)--(4,0);
\node [left] at (2,3.5) {$\phi_2(T)$};
\node [left] at (1.4,1.4) {$\phi_3(T)$};
\node [above] at (3.8, 0.6) {$\phi_1(T)$};
\node [left,below] at (-0.25, 0.2) {$0$};
\node [below] at (4, 0) {$1$};
\node [below] at (3, 0) {$\frac{3}{4}$};
\node [left] at (0, 4) {$1$};
\node [left] at (0, 5) {$y$};
\node [below] at (5, 0) {$x$};

\end{tikzpicture}
\caption{$\phi_i(T)$ ($i=1,2,3$) in Example \ref{ex: example to show nE not equal nconv(E)}.}
\label{fig2}
\end{figure}
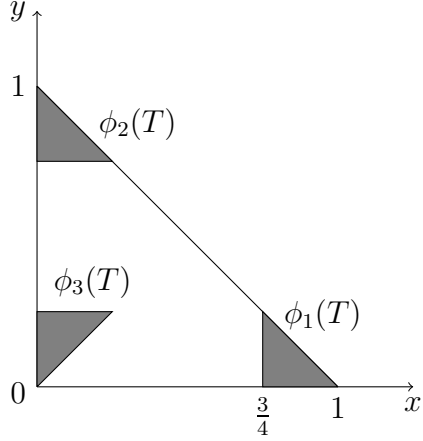

To see that  $\oplus_nE\neq nT$,  it is enough to show that $(\frac{1}{2}, 0)\not\in\oplus_nE$ for every $n\in \N$, since $(\frac{1}{2},0)\in nT$ for all $n$. To prove this, from Figure \ref{fig2} we observe that $E$ lies in the upper half plane, and the intersection of $E$ with the $x$-axis is contained in the set $(\{0\}\cup [3/4, 1])$ (in the first coordinate). It follows that each point in the intersection of $\oplus_nE$ with the $x$-axis has the first coordinate $0$ or $\geq \frac{3}{4}$. Hence $(\frac{1}{2}, 0)\not\in\oplus_nE$ for all $n\in \N$, as desired.
\end{proof}

 \section{Final remarks and questions}
 \label{S-8}
 In this section we give several remarks and questions.

First we remark that  the notion of thickness has certain robustness. Indeed, from Definition \ref{de-1.1} it is easy to see that if  $E\subset \R^d$ has positive thickness,  then so does the image of $E$ under any bi-Lipschitz map on $\R^d$. According to this fact and Lemmas \ref{lem-self-similar} and \ref{lem-4.1}, the image of an irreducible self-similar  (resp. self-conformal)  set in $\R^d$ ($d\geq 2$)  under any bi-Lipschitz map still has positive thickness and so is arithmetically thick by Theorem \ref{thm: sum of thick sets}. Here an irreducible self-similar set  means a self-similar set not lying in a hyperplane, whilst an irreducible self-conformal set means a self-conformal set in $\R^d$ that is not contained in any hyperplane or any $(d-1)$-dimensional sphere in the case when $d\geq 3$,   and is  not contained in an analytic curve in the case when $d=2$.

Secondly we can give  a very partial result on the arithmetic sums of  Ahlfors regular sets. Recall that a compact set $E\subset \R^d$ is said to be {\it Ahlfors $s$-regular} if there exist a finite Borel measure $\mu$ supported on $E$ and a constant $C\geq 1$ such that
\[r^s\leq \mu(B(x, r))\leq Cr^s\quad  \mbox{ for all } x\in E \mbox{ and } 0<r\leq {\rm diam}(E).\]
 It is not difficult to verify that every centred microset of an Ahlfors $s$-regular is again  Ahlfors $s$-regular (see e.g.~\cite[Lemma 9.7]{DavidSemmes1997}). Notice that an Ahlfors $s$-regular set  has Hausdorff dimension $s$. According to  Proposition \ref{prop: no mircoset in hyperlane implies thickness}, for every Ahlfors $s$-regular set $E\subset \R^d$ with $s>d-1$, $\tau(E)>0$ and so by Theorem \ref{thm: sum of thick sets}, $E$ is arithmetically thick.

Finally we pose a few questions.

{\noindent \bf Open Question 1}.  Is every self-affine set in $\R^d$ ($d\geq 3$)  arithmetically thick if it is not contained in a hyperplane in $\R^d$?

{\noindent \bf Open Question 2}. We do not have a good way to generalise our results to the arithmetic sums of the attractors of nonlinear non-conformal IFSs.   The challenge here is to analyse the local geometry and scaling properties of these fractal sets.
\medskip

{\noindent \bf  Acknowledgements}. This research was conducted as part of the second author's Ph.D. studies. It was partially supported by the HKRGC GRF grant (project 14301017) and the Direct Grant for Research in CUHK. The authors are grateful to Xiangyu Liang for pointing out the relation between the notion of thickness and that of uniform non-flatness, and to Ruojun Ruan for helpful discussions.

\end{document}